\documentclass[10pt, leqno, a4paper]{amsart}

\usepackage[latin1]{inputenc}
\usepackage{amsfonts}
\usepackage{amsmath}
\usepackage{amssymb}
\usepackage{amsthm}
\usepackage[T1]{fontenc}
\usepackage[dvips]{graphicx}
\usepackage{enumerate}
\usepackage{verbatim}

\usepackage[dvips]{hyperref}
\usepackage[matrix, arrow]{xy}

\input xy
\xyoption{all}
\SelectTips{cm}{10}

\DeclareMathOperator{\Proj}{Proj}

\DeclareMathOperator{\Syz}{Syz}
\DeclareMathOperator{\rk}{rk}
\DeclareMathOperator{\Rad}{Rad}

\DeclareMathOperator{\height}{ht}

\DeclareMathOperator{\chara}{char}
\DeclareMathOperator{\Spec}{Spec}
\DeclareMathOperator{\Pic}{Pic}
\DeclareMathOperator{\im}{im}

\newcommand{\eps}{\varepsilon}

\begin{document}
\swapnumbers
\theoremstyle{definition}
\newtheorem{Le}{Lemma}[section]
\newtheorem{Def}[Le]{Definition}
\newtheorem*{DefB}{Definition}
\newtheorem{Bem}[Le]{Remark}
\newtheorem{Ko}[Le]{Corollary}
\newtheorem{Theo}[Le]{Theorem}
\newtheorem*{TheoB}{Theorem}
\newtheorem{Bsp}[Le]{Example}
\newtheorem{Be}[Le]{Observation}
\newtheorem{Prop}[Le]{Proposition}
\newtheorem{Sit}[Le]{Situation}
\newtheorem{Que}[Le]{Question}
\newtheorem*{Con}{Conjecture}
\newtheorem{Dis}[Le]{Discussion}
\newtheorem{Prob}[Le]{Problem}
\newtheorem{Konv}[Le]{Convention}
\title{Dagger closure and solid closure in graded dimension two}
\author{Holger Brenner and Axel St\"abler}
\address{Universit\"at Osnabr\"uck, Fachbereich 6: Mathematik/Informatik,
Albrechtstr. 28a,
49069 Osnabr\"uck, Germany}
\email{hbrenner@uni-osnabrueck.de and axel.staebler@uni-osnabrueck.de}

\begin{abstract}
We introduce a graded version of dagger closure and prove that it coincides with solid closure for homogeneous ideals in two dimensional $\mathbb{N}$-graded domains of finite type over a field.
\end{abstract}
\maketitle
\section*{Introduction}
Both dagger closure and solid closure were introduced in the hope of providing a characteristic free notion of an ideal closure operation with properties similar to tight closure.

Dagger closure which is given in terms of valuations was first introduced by Hochster and Huneke in \cite{hochsterhunekedagger} for a complete local domain of positive characteristic. In their article they also proved that dagger closure coincides with tight closure in this setting [ibid.,Theorem 3.1]. We also note that Heitmann's \emph{full rank one closure} which he used to prove the direct summand conjecture in mixed characteristic in dimension $3$ (cf.\ \cite{heitmanndirectsummand}) is a variant of dagger closure tailored to mixed characteristics. Despite these striking results very little is known about dagger closure in equal characteristic zero.

Solid closure was introduced in \cite{hochstersolid} by Hochster.
The solid closure of an ideal is defined via local cohomology of so-called forcing algebras and reduction to the complete local domain case. Specifically, let $R$ be a noetherian ring and $f_1, \ldots, f_n, f_0 \in R$, then $f_0$ belongs to the \emph{solid closure} of $(f_1, \ldots, f_n)$, written $f_0 \in (f_1, \ldots, f_n)^\star$, if for every complete local domain $R' = \widehat{R_{\mathfrak{m}}}/ \mathfrak{q}$ (where $\mathfrak{m}$ is a maximal ideal of $R$ and $\mathfrak{q}$ is a minimal prime of $\widehat{R_{\mathfrak{m}}}$) we have $H^d_{\mathfrak{m}}(A') \neq 0$, where $d = \dim R'$ and $A' = R'[T_1, \ldots, T_n]/(f_1 T_1 + \ldots + f_n T_n + f_0)$ is the \emph{forcing algebra} for $(f_1, \ldots, f_n)$ and $f_0$ over $R'$.

In positive characteristic, solid closure coincides with tight closure under mild finiteness conditions (e.\,g.\ $R$ of finite type over a field -- see \cite[Paragraph 8]{hochstersolid}).
Solid closure in equal characteristic zero only yields a good closure operation in dimension less than three. This is due to an example by Roberts (cf.\ \cite{robertscomputation} or \cite[7.22 and 7.23]{hochstersolid}) showing that solid closure need not be trivial in regular rings of dimension $\geq 3$ containing the rationals\footnote{We also note that there is a refinement of solid closure called parasolid closure which has the right properties in equal characteristic zero in all dimensions (see \cite{brennerparasolid}).}.

In this paper we will introduce a graded version of dagger closure (denoted by $I^{\dagger \text{GR}}$ for an ideal $I$) and prove that this closure operation coincides with solid closure for $\mathbb{N}$-graded two dimensional affine domains over a field $k$. 

We prove in \cite[Corollary 2.12]{brennerstaeblerdaggerregular} as a Corollary to \cite[Theorem 3.1]{hochsterhunekedagger} that graded dagger closure agrees with tight closure in all dimensions for an $\mathbb{N}$-graded ring $R$ of finite type over a field $k = R_0$ of positive characteristic. Furthermore, we also prove in \cite[Corollary 3.9]{brennerstaeblerdaggerregular} that graded dagger closure does not admit the aforementioned pathology of solid closure. Namely, we show that graded dagger closure is trivial for polynomial rings over a field. Furthermore, in \cite{staeblerdaggerabelian} the second author proves an inclusion result for certain section rings of abelian varieties. This implies in particular that dagger closure is non-trivial in all dimensions.

In order to prove the equivalence of solid closure and graded dagger closure in dimension two we will use geometric interpretations of these closure operations in terms of vector bundles over the corresponding projective curves. Our main focus will be on characteristic zero although we will work in arbitrary characteristic.

For solid closure this geometric interpretation has been developed by the first author (cf.\ \cite{brennertightproj}, \cite{brennerslope} and \cite{brennertightplus}) which we now recall. 
Let $R$ be a normal standard graded domain of dimension $2$ over an algebraically closed field $k$ and write $Y = \Proj R$. Let $f_1, \ldots, f_n$ denote homogeneous generators of degrees $d_1, \ldots, d_n$ of an $R_+$-primary ideal and fix a homogeneous element $f_0$ of degree $m$. We may identify $H^0(Y, \mathcal{O}_Y(m))$ with $R_m$ (see \cite[Corollaire III.3.5]{SGA2} for a proof).

These data yield the following short exact sequence of locally free sheaves on $Y$ which we call the \emph{presenting sequence} for the twisted syzygy bundle \[\mathcal{S} = \Syz(f_1, \ldots, f_n)(m)\] with forcing data $(f_1, \ldots, f_n)$:
\[\begin{xy} \xymatrix{ 0 \ar[r] & \mathcal{S}\ar[r] & \bigoplus_{i=1}^n \mathcal{O}_Y(m - d_i) \ar[r]^<<<<<{f_1, \ldots, f_n} & \mathcal{O}_Y(m) \ar[r] & 0.}\end{xy}\]

The element $f_0 \in R_m$ defines via the connecting homomorphism the cohomology class \[c = \delta(f_0) \in H^1(Y, \Syz(f_1, \ldots, f_n)(m)).\] This class corresponds to the extension $\mathcal{S}' = \Syz(f_0, f_1, \ldots, f_n)(m)$ (see \cite[Ex.\ III.6.1]{hartshornealgebraic} for this correspondence). The complement $T = \mathbb{P}(\mathcal{S}'^\vee) \setminus \mathbb{P}(\mathcal{S}^\vee)$ is a geometric $\mathcal{S}$-torsor, which also corresponds to $c$.

Also note that in this situation $\mathbb{P}(\mathcal{S}^\vee)$ is a closed subvariety of $\mathbb{P}(\mathcal{S}'^\vee)$ and the (effective) Weil divisor corresponding to $s \in H^0(\mathbb{P}(\mathcal{S}'^\vee), \mathcal{O}(1)_{\mathbb{P}(\mathcal{S}'^\vee)}) = H^0(Y, \mathcal{S}^\vee)$ given by the dualised presenting sequence is precisely $\mathbb{P}(\mathcal{S}^\vee)$ (cf.\ \cite[Proposition 3.4 (iii)]{brennertightproj}) -- we call this the \emph{forcing divisor}.

The element $f_0$ is contained in the solid closure of $(f_1, \ldots, f_n)$ if and only if $T$ is not an affine scheme (cf.\ \cite[Proposition 3.9]{brennertightproj}). Note that $f_0$ is contained in the ideal if and only if $\delta(f_0)$ is zero.

If moreover, $\mathcal{S}$ is strongly semistable then non-affineness of $T$ is equivalent to $\mu(\mathcal{S}) \geq 0$ and $F^{e \ast}(c) \neq 0$ for all Frobenius pullbacks\footnote{If in this article some statement involves the Frobenius morphism without having explicitly fixed positive characteristic then the statement is also true in characteristic zero if one replaces the Frobenius by the identity.}. Or equivalently, $T$ is then non-affine if and only if $\mathcal{S}'^\vee$ is not ample (cf. \cite[Proposition 2.1]{brennertightplus}.

Beginning with section \ref{SectionAlmostZero} we will develop a machinery of almost zero for cohomology classes of vector bundles on curves. This constitutes the geometric interpretation of dagger closure (cf.\ Theorem \ref{almostzerodagger}) which will be crucial for proving that graded dagger closure and solid closure coincide in the two dimensional graded situation. We note that there are more elementary proofs in positive characteristic or if the base curve $Y$ is elliptic. But since this article is already quite long we decided not to include these here but refer the interested reader to \cite{staeblerthesis} instead.

The starting point of \cite{staeblerthesis} was that the first author suggested that dagger closure should somehow be related to the ampleness criterion of Seshadri (\cite[Theorem I.7.1]{hartshorneamplesubvarieties}). Namely, that curves intersecting the forcing divisor in a small way should correspond to syzygies of small order with respect to a valuation in an absolute integral closure of $R$.

As it turns out \cite[Theorem 2.3]{brennerslope}, which is itself a variant of Seshadri's criterion, is more suited to our needs. The heuristic is that if $\varphi: Y' \to Y$ is a finite dominant morphism of smooth curves, where $Y = \Proj R$, and $\mathcal{L}$ a line bundle on $Y'$ such that $\varphi^\ast \mathcal{S}'^\vee \to \mathcal{L}$ is surjective then one should be able to construct syzygies of $(f_0, \ldots, f_n)$ (in some suitable finite graded ring extension $S$ of $R$) from the dualised sequence. Furthermore, if $\deg \mathcal{L}/\deg \varphi$ is small then the degree of the syzygy 
should be small as well. The ring $S$ should be obtained via a section ring constructed from a suitable line bundle on $Y'$. We will need that these syzygies are of the form $(a_0, \ldots, a_n)$ with $a_0 \neq 0$.
 This means that the curve to which the surjection onto $\mathcal{L}$ corresponds is not contained in the forcing divisor -- we will refer to this non-containment in $\mathbb{P}(\mathcal{S}^\vee)$ as the \emph{support condition} in the following.

Proving that this is possible in the strongly semistable case if $f_0 \in I^\star$ will take up most of the article and is one instance where the notion of almost zero plays a crucial role. In fact, a cohomology class $c \in H^1(Y, \mathcal{S})$ is almost zero if and only if there are curves contradicting the ampleness of $\mathcal{S}'^\vee$ in the sense of \cite[Theorem 2.3]{brennerslope} that satisfy the support condition (see Theorem \ref{Lazas}).

We remark that one can prove the inclusion $I^{\dagger \text{GR}} \subseteq I^\star$ in the strongly semistable case without relying on ``almost zero''-techniques -- cf. \cite[Section 4]{staeblerthesis}. But to prove the other inclusion (even in the strongly semistable case) we heavily rely on the machinery of almost zero.

The strategy will be to prove the equivalence of these geometric notions in the case where the syzygy bundle is strongly semistable. This will be accomplished in Theorem \ref{Daggersolidsst}. Then in Section \ref{Reductions} we will use a strong Harder-Narasimhan filtration to extend this to arbitrary syzygy bundles. In Section \ref{AlgebraicReductions} we will relax our conditions on the ring and remove the primary condition on the ideal.

Our main result is Theorem \ref{Mainresult}:

\begin{TheoB}
Let $R$ denote an $\mathbb{N}$-graded two-dimensional domain of finite type over a field $R_0$ and $I$ a homogeneous ideal of $R$. Then the solid closure of $I$ coincides with the graded dagger closure of $I$. 
\end{TheoB}

In their article \cite{robertssinghannihilators} Roberts, Singh and Srinivas defined a notion of ``almost zero'' for modules and used this to study the notion of ``almost Cohen-Macaulay''. Their definition of almost zero is equivalent to our definition whenever both are applicable (see Remark \ref{AllesfixAnnulierbar} below). Also note that we will recover one of their results in dimension two in the case of an algebraically closed base field with our definition (cf.\ Proposition \ref{paramaz}). Their proof requires characteristic zero while our proof will be characteristic free.

This article is based on parts of the Ph.\,D.\ thesis of the second author -- see \cite{staeblerthesis}. In particular, the notion of almost zero for vector bundles and the issues around the support condition were largely developed by the second author.

\section{Preliminaries on strongly semistable vector bundles}

We recall that for a vector bundle $\mathcal{S}$ on a smooth projective curve $Y$ the \emph{slope} $\mu(\mathcal{S})$ is given by $\deg \mathcal{S}/\rk \mathcal{S}$. The vector bundle $\mathcal{S}$ is called semistable if for all locally free quotient sheaves $\mathcal{S} \to \mathcal{Q} \to 0$ one has $\mu(\mathcal{S}) \leq \mu(\mathcal{Q})$. This is equivalent to $\mu(\mathcal{E}) \leq \mu(\mathcal{S})$ for all inclusions $0 \to \mathcal{E} \to \mathcal{S}$.

Every locally free sheaf $\mathcal{S}$ has a unique \emph{Harder-Narasimhan filtration}. 
This is a filtration \[0 =\mathcal{S}_0 \subset \mathcal{S}_1 \subset \ldots \subset \mathcal{S}_t = \mathcal{S}\] of subbundles such that the quotients $\mathcal{S}_i / \mathcal{S}_{i-1}$ are locally free and semistable with slopes $\mu_i$ and such that $\mu_1 > \ldots > \mu_t$. One defines $\mu_{\max}(\mathcal{S}) = \mu_1$ and $\mu_{\min}(\mathcal{S}) = \mu_t$ the \emph{maximal slope} and the \emph{minimal slope}. This is the same as \[ \max \{\mu(\mathcal{E}) \, \vert \, 0 \to \mathcal{E} \to \mathcal{F} \text{ is a subsheaf of rank }\geq 1 \}\] and \[\min \{\mu(\mathcal{Q}) \, \vert \, \mathcal{F} \to \mathcal{Q} \to 0 \text{ is a locally free quotient sheaf of rank } \geq 1 \}\] respectively.

For a finite separable morphism $\varphi: Y' \to Y$ of smooth curves $\varphi^\ast \mathcal{S}$ is semistable if and only if $\mathcal{S}$ is semistable. In particular, the pullback of a Harder-Narasimhan filtration of $\mathcal{S}$ along $\varphi$ is the Harder-Narasimhan filtration of $\varphi^\ast \mathcal{S}$. We refer to \cite[I.5]{lepotier} for background on these notions.

Neither of this is true for inseparable morphisms. Hence, in positive characteristic the notion of semistability needs to be refined.
A locally free sheaf $\mathcal{S}$ on $Y$ is called \emph{strongly semistable} if for ever finite morphism $\varphi:Y' \to Y$ the pull back $\varphi^\ast \mathcal{S}$ is semistable\footnote{In particular, semistability and strong semistability coincide in characteristic zero.}. This is equivalent to $\mathcal{S}$ being semistable and that every pull back of $\mathcal{S}$ along the relative Frobenius is again semistable (cf.\ \cite[Proposition 5.1]{miyaokachern}).
By a theorem of Langer (\cite[Theorem 2.7]{langersemistable}) there is for given $\mathcal{S}$ an $e \geq 0$ such that the quotients of the Harder-Narasimhan filtration of $F^{e\ast} \mathcal{S}$ are strongly semistable. We call this a \emph{strong Harder-Narasimhan filtration} of $\mathcal{S}$. We write $\bar{\mu}_{\min}(\mathcal{S}) = \min\{ \mu_{\min}(\varphi^\ast \mathcal{S})/\deg \varphi\ \, \vert \, \varphi $ a finite dominant $k$-linear morphism$\}$ and similarly for $\bar{\mu}_{\max}(\mathcal{S})$. In particular, if the Harder-Narasimhan filtration of $\varphi^\ast \mathcal{S}$ has strongly semistable quotients then $\mu_{\min}(\varphi^\ast \mathcal{S})/ \deg \varphi = \bar{\mu}_{\min}(\mathcal{S})$.

\section{Graded dagger closure}

Recall that the \emph{absolute integral closure} $R^+$ of a domain $R$ is the integral closure of $R$ in an algebraic closure of $Q(R)$. This was first studied by M. Artin in \cite{artinabsintegral}.

In \cite{hochsterhunekedagger}, Hochster and Huneke gave a characterisation of tight closure in complete local domains of characteristic $p >0$ in terms of multipliers of small order with respect to a $\mathbb{Q}$-valued valuation $\nu$ on $R^+$.
In this situation any two valuations on $R$ which are positive on $\mathfrak{m}$ and non-negative on $R$ are equivalent by a theorem of Izumi (see \cite{izumimeasure}).

In the graded setting there is a canoncial choice for a valuation.
Namely, if $R$ is a $\mathbb{Q}$-graded domain the map $\nu: R \setminus \{ 0 \} \to \mathbb{Q}$ sending $f \in R \setminus \{0\}$ to $\deg f_i$, where $f_i$ is the minimal homogeneous component of $f$, induces a valuation on $R$ with values in $\mathbb{Q}$. This valuation will be referred to as the \emph{valuation induced by the grading}.

In order to define a graded version of dagger closure we need a graded version of $R^+$. This is provided by a result of Hochster and Huneke (\cite[Lemma 4.1]{hochsterhunekeinfinitebig}) which states that for an $\mathbb{N}$-graded domain $R$ there is a maximal $\mathbb{Q}_{\geq 0}$-graded subring of $R^+$ which extends the grading of $R$ -- we will denote this ring by $R^{+ \text{GR}}$. It is the limit of all $\mathbb{Q}_{\geq 0}$-graded integral extension domains of $R$.

\begin{Def}
\label{Defgradedagger}
Let $R$ denote an $\mathbb{N}$-graded domain and let $I$ be an ideal of $R$. Let $\nu$ be the valuation on $R^{+ \text{GR}}$ induced by the grading. Then we define the \emph{graded dagger closure} $I^{\dagger \text{GR}}$ of an ideal $I$ as the set of elements $f$ in $R$ such that for all positive $\varepsilon$ there exists an element $a \in R^{+ \text{GR}}$ with $\nu(a) < \varepsilon$ and such that $af$ lies in the extended ideal $IR^{+ \text{GR}}$.
\end{Def}

\begin{Le}
\label{daggerhomogeneous0}
 The graded dagger closure of a homogeneous ideal is again a homogeneous ideal.
\end{Le}
\begin{proof}
Standard.
\end{proof}

This result also implies that one may choose the multipliers of small order in Definition \ref{Defgradedagger} to be homogeneous. As another immediate consequence we have that an element $f$ is contained in $I^{\dagger \text{GR}}$ if and only if all of its homogeneous components are contained in $I^{\dagger \text{GR}}$. We will therefore restrict our attention to homogeneous elements.

\begin{Le}
\label{Ldeg}
Let $R$ be an $\mathbb{N}$-graded domain and assume that the grading is non-trivial. An element $f$ belongs to $I^{\dagger \text{GR}}$ if and only if there is a sequence $a_n$ of elements in $R^{+ \text{GR}}$  with $\nu(a_n) = 1/n$ such that $a_n f \in IR^{+ \text{GR}}$.
\end{Le}
\begin{proof}
Only the only if part is non-trivial. So let $a$ be an element with $\deg(a) \leq \frac{1}{n}$ such that $af \in IR^{+ \text{GR}}$. We only need to find an element $x$ of degree $\frac{1}{n} - \deg(a)$, because we have $xaf \in IR^{+ \text{GR}}$ since this is an ideal.
Let $b$ be a homogeneous element of $\deg b = l > 0$. Fix $m \in \mathbb{N}$ and consider the polynomial $X^{ln} - b^m$, the zeros of this polynomial are homogeneous elements of degree $\frac{m}{n}$. Thus we can construct elements of arbitrary positive degree in $R^{+ \text{GR}}$.
\end{proof}

It is quite obvious from the definition that the tight closure of a homogeneous ideal is again homogeneous. This is not immediate for solid closure in characteristic zero. Thus the following

\begin{Prop}
\label{solidhomogeneouscharazero}
Let $R$ be an $\mathbb{N}$-graded ring, where $R_0$ is a field of characteristic zero which contains all roots of unity, and let $I$ be a homogeneous ideal. Then $I^\star$ is homogeneous.
\end{Prop}
\begin{proof}
Consider the ring automorphisms $\varphi_{\lambda}$ for $\lambda \in R_0^\times$ which map a homogeneous element $x$ to $\lambda^{\deg x} x$. Since $I$ is homogeneous we have $\varphi_{\lambda}(I) = I$ and consequently $I^\star = \varphi_{\lambda}(I^\star)$ as solid closure is derived in an algebraic way from the ideal. Assume that $f \in I^\star$. We have to show that each homogeneous component of $f$ is contained in $I^\star$. We will induct on the number $r$ of nonzero components. The assertion is clear for $r = 0,1$. So let $r > 1$.

Write $f = \sum_{i=0}^n f_i$ where $\deg f_i = i$ and assume that $f_n \neq 0$. Let $\lambda$ be a primitive $n$th root of unity. Then $\varphi_{\lambda}(f) - f \in I^\star$ and this has $r-1$ nonzero homogeneous components. Hence, the $(\lambda -1) f_i$ belong to $I^\star$ and therefore also $f_n$.
\end{proof}

\begin{Bem}
\label{solidhomogeneouscharazeroremark}
Proposition \ref{solidhomogeneouscharazero} applies more generally to any ideal closure operation which is associated to the ideal in an algebraic way, even when extended to graded modules. Moreover, the condition that $R_0$ contains all roots of unity is not very restrictive. Indeed, assume that $\phantom{}^\ast$ is an ideal closure operation, $\varphi: R \to S$ a ring homomorphism, $I \subseteq R$ an ideal and $f \in (\varphi(I)S)^\ast$. Suppose furthermore that  $\phantom{}^\ast$ satisfies one of the  following properties:
\begin{enumerate}[(i)]
 \item{If $R \subseteq S$ is an integral ring extension then $(IS)^\ast \cap R \subseteq I^\ast$.}
 \item{If $S$ is faithfully flat over $R$ then $f \in I^\ast$.}
\end{enumerate}
Then we can drop the assumption that $R_0$ contains all roots of unity. Indeed, if $k$ is an algebraic extension which contains all roots of unity of $R_0$ then $R_0 \subseteq k$ is integral and $R \otimes_{R_0} k$ is faithfully flat. Also note that solid closure satisfies this second condition (see \cite[Theorem 5.9]{hochstersolid}).
\end{Bem}

For ease of reference we thus have

\begin{Ko}
\label{solidhomogeneous}
Let $R$ be an $\mathbb{N}$-graded domain of finite type over a field $R_0$ and let $I$ be a homogeneous ideal. Then $I^\star$ is again homogeneous.
\end{Ko}
\begin{proof}
If the characteristic of $R$ is zero then the result is due to Proposition \ref{solidhomogeneouscharazero} and Remark \ref{solidhomogeneouscharazeroremark}. If  $\chara R > 0$ then $I^\star = I^\ast$ by virtue of \cite[Theorems 8.5 and 8.6]{hochstersolid}.
\end{proof}

\begin{Def}
\label{moddag}
Let $R$ be an $\mathbb{N}$-graded domain and let $N \subseteq M$ be an inclusion of $\mathbb{Z}$-graded $R$-modules. Then the \emph{graded dagger closure} of $N$ in $M$ denoted by $N_M^{\dagger \text{GR}}$ is the set of all elements $m \in M$ such that for all $\eps > 0$ there exists $a \in R^{+ \text{GR}}$ with $\nu(a) < \eps$ and $m \otimes a_n \in \im(N \otimes R^{+ \text{GR}} \to M \otimes R^{+ \text{GR}})$. As usual, $\nu$ is the valuation induced by the grading.
\end{Def}

\begin{Le}
\label{daggermodulistdagger}
Let $R$ be an $\mathbb{N}$-graded domain and $I \subseteq R$ an ideal.
Then $I_R^{\dagger \text{GR}} = I^{\dagger \text{GR}}$, where the latter denotes the usual graded dagger closure as in Definition \ref{Defgradedagger}.
\end{Le}
\begin{proof}
Consider the exact sequence $0 \to I \to R \to R/I \to 0$ and tensor with $R^{+ \text{GR}}$ to obtain a right exact sequence \[I \otimes R^{+ \text{GR}} \to R^{+ \text{GR}} \to R^{+ \text{GR}}/IR^{+ \text{GR}} \to 0.\] Assume $f \in I^{\dagger \text{GR}}$. Then 
for every $\eps >0$ we find an element $a \in R^{+ \text{GR}}$ with $\nu(a) < \eps$ sucht that $af \in IR^{+ \text{GR}}$. Since the canonical map $I \otimes R^{+  \text{GR}} \to IR^{+ \text{GR}}$ is surjective we have preimages inside $I \otimes R^{+ \text{GR}}$.

For the other inclusion we have that the $af \in R^{+ \text{GR}}$ have preimages in $I \otimes R^{+ \text{GR}}$. That is, they map to zero in $R^{+ \text{GR}}/IR^{+ \text{GR}}$. Hence, $af \in IR^{+ \text{GR}}$.
\end{proof}

We note that dagger closure for modules is also considered in \cite{asgbhattaremakrsonbcm}.

\begin{Prop}
\label{Pwhatever}
In the situation of Definition \ref{moddag} consider the short exact sequence $0 \xrightarrow{} N \xrightarrow{} M \xrightarrow{p} M/N \xrightarrow{} 0$ of $\mathbb{Z}$-graded $R$-modules. Then $m \in N_M^{\dagger \text{GR}}$ if and only if $p(m) \in 0_{M/N}^{\dagger \text{GR}}$.
\end{Prop}
\begin{proof}
Follows similarly to Lemma \ref{daggermodulistdagger}.
\end{proof}

Proposition \ref{Pwhatever} motivates the following

\begin{Def}
\label{AlmostZeroGradedModules}
Let $R$ be an $\mathbb{N}$-graded domain and $M$ a graded $R$-module. We say that $m \in M$ is \emph{almost zero} if for every $\eps > 0$ the element $m \otimes 1 \in M \otimes R^{+ \text{GR}}$ is annihilated by an element $a_m \in R^{+ \text{GR}}$ with $\nu(a_m) < \eps$ (equivalently $m \in 0^{\dagger \text{GR}}_M$).
\end{Def}

\begin{Prop}
\label{NetallesAz}
Let $R$ be an $\mathbb{N}$-graded domain finitely generated over a field $R_0$ and $M \neq \{0\}$ a finitely generated $\mathbb{Z}$-graded $R$-module. Then not every element of $M$ is almost zero in the sense of Definition \ref{AlmostZeroGradedModules}.
\end{Prop}
\begin{proof}
If $M$ is generated by a single element then we have a presentation $0 \to I \to R \to M \to 0$ and we may assume that $M = R/I$, where $I$ is homogeneous. If $M$ is free so that $I = 0$ then $M \otimes_R R^{+ \text{GR}} = R^{+ \text{GR}}$ and clearly this module is not almost zero. 

So assume that $I \neq 0$. If the characteristic of $R$ is positive then we have that $I^{\dagger \text{GR}} \subseteq I^\ast \subseteq \Rad(I)$ and hence $R/I$ is not almost zero by Proposition \ref{Pwhatever}. If $R$ is of characteristic zero then by \cite[Corollary 4.6]{brennerstaeblerdaggerregular} any homogeneous ideal $J$ of height $\dim R$ is contained in $\Rad(J)$. It follows that if $I \subseteq J$ and every element of $R/I$ were almost zero then also every element of $R/J$. 

Assume that $M$ is minimally generated by homogeneous elements $m_1, \ldots, m_n$, $n \geq 2$ and assume that every element of $M$ is almost zero. Consider the short exact sequence $0 \to N \to M \to M/N \to 0$, where $N$ denotes the submodule of $M$ generated by $m_1, \ldots, m_{n-1}$. Then since every element of $M$ is almost zero so is every element of $M/N$. But $M/N$ is generated by the class of $m_n$. This contradiction proves the proposition.
 \end{proof} 

\section{Some useful lemmata for section rings}
In this section we collect some lemmata that we shall need frequently in the following sections.

\begin{Le}
\label{Pullback}
Let $R$ be an $\mathbb{N}$-graded domain such that $\Proj R$ is covered by open sets $D_+(f)$, where $f \in R_1$, and let $R \subseteq S$ be a finite $\mathbb{Q}$-graded extension of domains. Consider $S$ as an $\mathbb{N}$-graded domain by multiplying with a common denominator $n$. Then the inclusion induces a morphism $\varphi: \Proj S \to \Proj R$ and we have $\varphi^\ast \mathcal{O}_{\Proj R}(1) = \mathcal{O}_{\Proj S}(n)$.
\end{Le}
\begin{proof}
For the construction of the induced morphism confer \cite[2.8.1]{EGAII}. To see that the morphism is defined on all of $\Proj S$ we have to prove that $V_+(\varphi(R_+)) = \varnothing$. Write $S = R[x_1, \ldots, x_m]$. Each $x = x_j$ satisfies an equation $x^d = \sum_{i=0}^{d-1} a_i x^i$, where the $a_i \in R$ are homogeneous, and the degree of $x$ in $S$ is $\frac{\deg a_i}{d-i}$. If $\deg x=0$ then $x \notin S_+$. We have $\deg x > 0$ if and only if $\deg a_i > 0$ for each $i$. Hence $x^d \in S R_+$. 
This shows that $\Rad(S R_+) = S_+$ and hence $V_+(\varphi(R_+)) = \varnothing$. 

We can cover $\Proj R$ by open subsets of the form $D_+(f)$, where $f$ has degree $1$. Hence, $\Proj S$ is covered by $\varphi^{-1}(D_+(f)) = D_+(i(f))$ where $i: R \to S$ is the inclusion (cf.\ \cite[2.8.1]{EGAII}). The covering property also implies that $\mathcal{O}_{\Proj R}(1)$ is invertible (cf.\ \cite[proof of Proposition II.5.12 (a)]{hartshornealgebraic}).
It is enough to show that  $\varphi^\ast \mathcal{O}_{\Proj R}(1)$ is isomorphic to $\mathcal{O}_{\Proj S}(n)$ if we restrict to open subsets of the form $D_+(i(f))$ and that these isomorphisms are consistent on twofold intersections.
Indeed, locally we have $R(1)_{(f)} = fR_{(f)}$ since $\deg f = 1$ and $f$ is invertible in $R_{f}$. Moreover, \[\varphi^\ast \mathcal{O}_{\Proj R}(1)\vert_{D_+(i(f))}= (f R_{(f)} \otimes_{R_{(f)}} S_{(f)})^{\sim} \text{ and } f R_{(f)} \otimes_{R_{(f)}} S_{(f)} \cong f S_{(f)}\] which is immediately verified using the universal property of the tensor product. 
And these are precisely the elements of degree $n$ in $S_{f}$.

We now show that this is consistent on twofold intersections. If $g$ is another element of degree $1$ then we have $\varphi^\ast\mathcal{O}(1) = (R_{(fg)} (1) \otimes_{R_{(fg)}} S_{(fg)})^\sim$ on $D_+(fg)$ and $R_{(fg)} = f R(1)_{(fg)}$. Arguing as above thus implies $R_{(fg)}(1) \otimes_{R_{(fg)}} S_{(fg)} \cong S(n)_{(fg)}$.
\end{proof}

 At this point we should probably note that there is a more geometric (and more general) version of the previous lemma. Namely the following

\begin{Le}
\label{allgPullback}
 Let $X$ be a projective variety. Fix an ample line bundle $\mathcal{O}(1)$ and a line bundle $\mathcal{L}$. Let $m$ be such that $\mathcal{L}(m)$ is generated by global sections and choose $N \in \mathbb{N}$. Then there exists a finite surjective morphism $f:Y \to X$ where $Y$ is a normal projective variety such that $f^\ast \mathcal{L}(m) = \mathcal{M}^{\otimes N}$, where $\mathcal{M}$ is a line bundle which is generated by global sections.
\end{Le}
\begin{proof}
 This is covered in \cite[proof of Theorem 3.4]{robertssinghannihilators}.
\end{proof}

\begin{Def}
Let $f: Y' \to Y$ be a finite dominant morphism of projective varieties and $\mathcal{L}$ a line bundle on $Y$. We call a line bundle $\mathcal{M}$ on $Y'$ an \emph{$n$th root} (or a \emph{root}) of $\mathcal{L}$ if $\mathcal{M}^n$ is isomorphic to $f^\ast\mathcal{L}$, where $n \in \mathbb{N}$.
\end{Def}

\begin{Le}
\label{stdgraded}
Let $R$ be an $\mathbb{N}$-graded domain finitely generated over a field $k = R_0$. Then there exists a finite ring extension $R \subseteq S$ such that $S$ is a standard graded domain.
\end{Le}
\begin{proof}
We may assume $R$ to be normal. Fix homogeneous algebra generators $r_1, \ldots, r_n$ of $R$ and write $d_i = \deg r_i$. Consider the polynomials $f_i = X^{d_i} - r_i$. Fix an irreducible polynomial $g_1$ dividing $f_1$. By \cite[Corollary 4.12]{Eisenbud}, $g_1$ is prime in $R[X]$ and $S_1 :=R[X]/(g_1)$ is a finite extension domain such that $f_1$ has a root, i.\,e.\ a (homogeneous) element $x \in S_1$ such that $x^{d_1} = r_1$. Normalising $S_1$ and repeating this process we obtain a finite $\mathbb{N}$-graded extension domain $S'$ where each $r_i$ has a $d_i$th root which we call $x_i$. Then $S = k[x_1, \ldots, x_n]$ is of the desired form.
\end{proof}

Recall that the \emph{section ring} of an invertible sheaf $\mathcal{L}$ on a scheme $X$ is defined to be the graded ring $\bigoplus_{n \geq 0}\Gamma(X, \mathcal{L}^n)$.

Section rings are much better behaved then arbitrary graded rings. In particular, if $S$ is the section ring of an ample invertible sheaf on a projective variety then $S(1)^\sim$ is invertible and equal to $\mathcal{L}$. Moreover, if the variety is normal then so is $S$ (see e.\,g.\ \cite[Proposition 2.1]{hyrysmith}).

Another issue in working with arbitrary graded extension domains $R \subseteq S$ is that we cannot ensure that $S$ is standard graded and normal -- the normalisation of a standard graded ring may no longer be standard graded. However, if $S$ is the section ring of a globally generated ample line bundle $\mathcal{L}$ then $\Proj S$ is covered by standard open sets $D_+(f)$, where the $f \in S$ are of degree $1$. This slightly weaker condition is stable under finite pullbacks and is a good enough replacement for standard graded in virtually every situation. Also recall that if $\mathcal{L}$ is an ample invertible sheaf on a projective scheme $X$ and $S$ the section ring of $\mathcal{L}$ then one has a canonical isomorphism $X \to \Proj S$ -- see \cite[Th\'{e}or\`{e}me 4.5.2]{EGAII}.

The following Lemma is contained in \cite[Lemma 3.10]{brennertightproj} but the author does not provide a proof.

\begin{Le}
\label{Leholger}
Let $f: Y' \to Y$ be a finite dominant morphism of projective varieties over a field $k$. Fix an ample line bundle $\mathcal{O}(1)$ on $Y$. Then \[\bigoplus_{n \geq 0} \Gamma(Y, \mathcal{O}(n)) \subseteq \bigoplus_{n \geq 0} \Gamma(Y', f^\ast\mathcal{O}(n))\]
is a finite extension of graded domains. 
\end{Le}
\begin{proof}
Call these rings $R$ and $S$. 
To begin with, $R$ and $S$ are domains. For if $s, t \in S$ are nonzero and homogeneous, where $s \in \Gamma(Y', f^\ast\mathcal{O}(m)), t \in \Gamma(Y', f^\ast\mathcal{O}(n))$, then multiplication $s \cdot t$ corresponds to a morphism $f^\ast\mathcal{O}(n) \to f^\ast\mathcal{O}(n) \otimes f^\ast\mathcal{O}(m) = f^\ast \mathcal{O}(n+m)$ induced by $\mathcal{O} \to f^\ast\mathcal{O}(m)$. The latter morphism is injective and since tensoring with $f^\ast\mathcal{O}(n)$ is exact we have that $S$ is a domain. One shows similarly that $R$ is a domain.

Note that $f^\ast\mathcal{O}(1)$ is ample by \cite[Ex.\ III.5.7 (d)]{hartshornealgebraic} or \cite[Proposition 4.4]{hartshorneamplesubvarieties} (the morphism is dominant hence surjective since it is proper).
 As the line bundles are ample, $R$ and $S$ are $k$-algebras of finite type (see e.\,g.\ \cite[Proposition 9.2]{badalgsurf}).
We have injective morphisms $\mathcal{O}(n) \to f_\ast f^\ast \mathcal{O}(n)$ (this follows since $\mathcal{O}_{f(y)} \to f_\ast \mathcal{O}_y$ is injective for $y \in Y'$, cf. \cite[Ex.\ I.3.3 (c)]{hartshornealgebraic}). 
Moreover, we have a commutative diagram
\[\begin{xy}
   \xymatrix{\mathcal{O}(n) \otimes \mathcal{O}(m) \ar[d] \ar[r] & \mathcal{O}(n+m)\ar[d]\\
f_\ast f^\ast \mathcal{O}(n) \otimes f_\ast f^\ast \mathcal{O}(m) \ar[r] & f_\ast f^\ast \mathcal{O}(n+m).}
  \end{xy}\]
Taking global sections shows that we have an injective $k$-linear ringhomomorphism as desired. See also \cite[Chapitre 0, 5.4.6]{EGAI}.

By homogeneous Noether normalisation (\cite[Theorem 1.5.17]{brunsherzog}) we find homogeneous elements $x_0, \ldots, x_d$ of $R$ such that the $D_+(x_i)$ cover $\Proj R = Y$. Then we have that the $f^{-1}(D_+(x_0)) = D_+(i(x_0)), \ldots, f^{-1}(D_+(x_d)) = D_+(i(x_d))$ cover $\Proj S$, where $i: R \to S$ denotes the inclusion. Thus, again by homogeneous Noether normalisation, $R \subseteq S$ is finite.
\end{proof}

\begin{Prop}
\label{Propringext}
Let $f: Y' \to Y$ be a finite dominant morphism of projective varieties and let $\mathcal{O}_Y(1)$ be an ample line bundle on $Y$. Fix an ample line bundle $\mathcal{L}$ on $Y'$ such that $\mathcal{L}^k = f^\ast\mathcal{O}_Y(l)$ for some $k, l$ in $\mathbb{N}$. Choose a minimal such $k$ and identify these two line bundles along a fixed isomorphism. Then \[S = \bigoplus_{n,m \geq 0} \Gamma(Y', \mathcal{L}^m \otimes f^\ast \mathcal{O}_Y(n))/\sim\] is a finite extension domain of $R =\bigoplus_{n \geq 0} \Gamma(Y, \mathcal{O}_Y(n))$, where $\sim$ denotes the identification made above. Moreover, it is a graded ring extension if we assign elements of $\mathcal{L}^m(n)$ the degree $\frac{ml}{k} + n$.
\end{Prop}
\begin{proof}
By Lemma \ref{Leholger} we have that $\bigoplus_{n \geq 0} \Gamma(Y, \mathcal{O}_Y(n)) \subseteq \bigoplus_{n \geq 0} \Gamma(Y', f^\ast\mathcal{O}_Y(n))$ is a finite extension of graded domains. In particular, the original extension $R \subseteq S$ is integral.
To prove that the extension is finite it remains to show that \[\bigoplus_{t \geq 0}\Gamma(Y', \mathcal{L}^s(t)) \text{ is finite over } \bigoplus_{n \geq 0}\Gamma(Y', f^\ast \mathcal{O}_Y(n)) =: T \text{ for } s=1, \ldots, k-1.\]

To this end, we may replace $Y'$ by its normalisation. By \cite[Ex.\ II.5.14]{hartshornealgebraic}, a $d$-uple embedding $T^{(d)}$ of $T$ is then projectively normal for $d \gg 0$. Note that $T$ is finite over $T^{(d)}$ (use \cite[Corollary 4.5]{Eisenbud}) and that $T^{(d)}$ is standard graded. This in turn allows us to invoke \cite[Ex.\ II.5.9]{hartshornealgebraic}. Thus for each $s$ we have finitely many $\Gamma(Y', \mathcal{L}^s(t))$ which are finite dimensional vector spaces and a finite module. Hence the extension is finite. 
\end{proof}



\begin{Prop}
\label{PCover}
Let $Y$ be a projective variety with an ample line bundle $\mathcal{L}$. Then $\Proj \bigoplus_{n\geq 0} \Gamma(Y, \mathcal{L}^n)$ can be covered by a finite number of $D_+(f)$ where $f \in \Gamma(Y, \mathcal{L})$ if and only if $\mathcal{L}$ is generated by global sections.
\end{Prop}
\begin{proof}
We will use the canonical isomorphism $Y \to \Proj \bigoplus_{n\geq0}\Gamma(Y, \mathcal{L}^n)$ to identify $Y$ and $\Proj \bigoplus_{n\geq0}\Gamma(Y, \mathcal{L}^n)$.

First of all, we show that $Y_{f} = \{y \in Y : f \notin \mathfrak{m}_y \mathcal{L}_y \}$ and $D_+(f) = \{y \in Y : f \notin y\}$ are equal for $f \in \Gamma(Y, \mathcal{L})$. See also \cite[Proposition 2.6.3]{EGAII} (especially for the second inclusion).

Fix a homogeneous prime $P$ in $S = \bigoplus_{n\geq 0} \Gamma(Y, \mathcal{L}^n)$ and assume that $f \notin P$, where $\deg f =1$. Then $f$ is a generator of $\Gamma(D_+(f),\mathcal{L}) = S(1)_{(f)}$. And hence $f$ is a generator of $\mathcal{L}_{P}$. 
If $f \in \mathfrak{m}_P \mathcal{L}_P$ then $\mathfrak{m}_P \mathcal{L}_P = \mathcal{L}_P$ and we get a contradiction by Nakayama's lemma.

For the other inclusion observe that $D_+(g)$ for finitely many $g \in S_d$ cover $Y$ for some $d$ since $\mathcal{L}$ is ample (cf.\ \cite[Th\'{e}or\`{e}me 4.5.2]{EGAII}) and $Y$ is quasi-compact. Consider $P \in D_+(g) \cap Y_f$ where $\deg f = 1$ -- in particular, $f$ generates the stalk $\mathcal{L}_P$. 
Hence, both $g$ and $f^d$ are generators of $\mathcal{L}^d_P$ and we therefore have $\frac{g}{f^d} f^d = g$. Consequently, $f^d$ (and then $f$) is a unit in $S_P$ since $g$ is. And this means $f \notin P$.

We thus have $D_+(f) = Y_f$. Moreover, the $Y_f$ cover $Y$ if and only if $\mathcal{L}$ is generated by global sections. And since $Y$ is quasi-compact a finite number of them will do.
\end{proof}

\begin{Prop}
\label{Syzdet}
Let $Y$ denote a smooth projective curve over an algebraically closed field and let $\mathcal{S}$ denote a locally free sheaf of rank $r$. Assume furthermore that $\det \mathcal{S} = \mathcal{O}_Y(n)$ for some integer $n$, where $\mathcal{O}_Y(1)$ is an ample invertible sheaf.
Then some twist of $\mathcal{S}$ is a syzygy bundle with respect to $R =\bigoplus_{l \geq 0} \Gamma(Y, \mathcal{O}_Y(l))$.
\end{Prop}
\begin{proof}
Since $\mathcal{O}_Y(1)$ is ample, we have that $\mathcal{S}^\vee(m)$ is generated by global sections for some $m \geq 0$. Hence, by \cite[Lemma 2.3]{brennertestexponent} we obtain an exact sequence $0 \to \mathcal{L} \to \mathcal{O}_Y^{r+1} \to \mathcal{S}^\vee(m) \to 0$, where $\mathcal{L}$ is a line bundle. By \cite[Ex.\ II.5.16 (d)]{hartshornealgebraic} we have an isomorphism $\mathcal{L} = \mathcal{O}_Y(-mr) \otimes (\det \mathcal{S}) = \mathcal{O}_Y(n - mr)$. Therefore, we get an exact sequence
$0 \to \mathcal{S} \to \mathcal{O}_Y(m)^{r+1} \to \mathcal{O}_Y(m(r+1) -n) \to 0$. Twisting by $\mathcal{O}_Y(n-m(r+1))$ yields that $\mathcal{S}(n -m(r+1))$ is a syzygy bundle.
\end{proof}

\section{Almost zero for vector bundles}
\label{SectionAlmostZero}

In this section we will define a notion of almost zero for cohomology classes of vector bundles on curves. We will compare this to the notion of almost zero in \cite{robertssinghannihilators} and prove that almost zero characterises dagger closure.

\begin{Def}
\label{almostzero}
Let $\mathcal{S}$ be a vector bundle on a smooth projective curve $Y$ over an algebraically closed field $k$ together with a cohomology class $c \in H^1(Y, \mathcal{S})$. We say that $c$ is \emph{almost zero} if for all $\eps > 0$ there exists a finite morphism $\varphi: Y' \to Y$ of smooth projective curves and a line bundle $\mathcal{L}$ on $Y'$ with a global section $s \neq 0$ such that $\deg \mathcal{L}/\deg \varphi < \eps$ and such that \[s \varphi^\ast(c) = 0 \in H^1(Y', \mathcal{L} \otimes \varphi^\ast \mathcal{S}).\] Here $s \varphi^\ast(c)$ is induced by the morphism \[\begin{xy}\xymatrix{0 \ar[r]& \varphi^\ast \mathcal{S} \ar[r]^>>>>>{\cdot s}& \varphi^\ast \mathcal{S} \otimes \mathcal{L}}.\end{xy}\] 
We say that $\mathcal{S}$ is almost zero if every $c \in H^1(Y, \mathcal{S})$ is almost zero. By abuse of notation we will also sometimes say that $s$ annihilates $c$ if $s \varphi^\ast(c) = 0$.
\end{Def}

\begin{Bem}
\label{BemAlmostZeroGedoens}
If $\mathcal{S}$ is almost zero then there exists for every $\eps >0$ a line bundle $\mathcal{L}$ with a global section $s \neq 0$ such that $\deg \mathcal{L}/\deg \varphi < \eps$ so that $s$ annihilates every $c \in H^1(Y,\mathcal{S})$. This follows since $H^1(Y,S)$ is finitely generated as a $k$-vector space. We note, however, that this neither implies that $H^1(Y', \varphi^\ast \mathcal{S}) =0$ nor that $H^1(Y', \varphi^\ast \mathcal{S}) \to H^1(Y', \varphi^\ast \mathcal{S} \otimes \mathcal{L})$ is the zero map.
\end{Bem}

We now want to prove that the notion of ``almost zero'' in the sense of \cite[Question 3.3]{robertssinghannihilators} implies our notion of ``almost zero'' over an algebraically closed field of characteristic zero.
First, we recall this definition in a form adapted to our situation.

\begin{Def}
\label{RSSalmostzero}
Let $R$ be an $\mathbb{N}$-graded domain which is finitely generated over a field $R_0$ of characteristic zero, write $X = \Proj R$ and assume that $\mathcal{O}_X(1)$ is invertible. An element $c \in H^1(X, \mathcal{O}_X(m))$ is \emph{almost zero} if for every $\eps > 0$ there exists a finite $\mathbb{Q}$-graded extension $R \subseteq S$ (preserving degrees) such that the image of $c$ under the induced map $H^1(X, \mathcal{O}_X(m)) \to H^1(Y, \varphi^\ast \mathcal{O}_X(m))$, where $\varphi: Y = \Proj S \to X$ (here the grading of $S$ is multiplied by an integer so that $S$ is $\mathbb{N}$-graded), is annihilated by an element of $S$ of degree $< \eps$.

We say that $\mathcal{O}_X(m)$ is almost zero if every $c \in H^1(X, \mathcal{O}_X(m))$ is almost zero.
\end{Def}

We remark that the definition in \cite{robertssinghannihilators} neither requires $\mathcal{O}_X(1)$ to be invertible nor $R$ to be normal but all applications given there reduce to the invertible case on the ring level. Likewise, we can assume in addition $R$ to be normal by passing to a finite extension domain.

\begin{Prop}
\label{RSSazimpliesours}
Let $R$ be a normal two-dimensional $\mathbb{N}$-graded domain which is finitely generated over an algebraically closed field $R_0$ of characteristic zero such that $\mathcal{O}_{X}(1)$ is invertible, where $X = \Proj R$. If $c \in H^1(X, \mathcal{O}_{X}(m))$ is almost zero in the sense of Definition \ref{RSSalmostzero} then it is also almost zero in the sense of Definition \ref{almostzero}.
\end{Prop}
\begin{proof}

Assume that $c \in H^1(X, \mathcal{O}_X(m))$ is almost zero. This means that for every $n \in \mathbb{N}$ there is a finite extension $R \subseteq S$ such that the image of $c$ under the map $H^1(X, \mathcal{O}_X(m)) \to H^1(Y, \varphi^\ast\mathcal{O}_X(m))$ is annihilated by 
some nonzero homogeneous element $s \in S$ of $\deg s \leq \frac{1}{n}$, where $Y = \Proj S$ and $\varphi: Y \to X$ is induced by the inclusion (here we regrade $S$ before taking the $\Proj$). Since $R$ is normal and of dimension $2$ we have $\Gamma(X, \mathcal{O}_X(n)) = R_n$.
Adjoining suitable roots to $R$ and then normalising again we may assume that $R$ is normal and that $\Proj R$ is covered by standard open sets coming from elements of degree $1$. Furthermore, we replace $S$ by the subring of $R^{+ \text{GR}}$ generated by the elements of $S$ and $R$.

As in the proof of Lemma \ref{Ldeg} we may assume that $\deg s = \frac{1}{n}$. Moreover, applying Lemma \ref{Pullback} we may assume that $S$ is generated by elements of degree $\frac{1}{n}$. Furthermore, $\varphi^\ast \mathcal{O}_X(1) = \mathcal{O}_{Y}(n)$. Consider the normalisation $S'$ of $S$ and denote the induced morphism by $\psi: \Proj S' = Y' \to Y$. The element $s$ then induces a morphism $\mathcal{O}_{Y'} \to \psi^\ast \mathcal{O}_{Y}(1)$ and tensoring with $\psi^\ast \varphi^\ast \mathcal{O}_X(m) = \psi^\ast \mathcal{O}_{Y}(nm)$, and taking cohomology yields a map that annihilates $c$. Furthermore, $\frac{\deg \psi^\ast \mathcal{O}_{Y}(1)}{\deg \varphi \psi} = \frac{\mathcal{O}_{Y}(1)}{\deg \varphi} = \frac{1}{n} \deg \mathcal{O}_X(1)$.
\end{proof}

See Remark \ref{AllesfixAnnulierbar} as to what extend the converse holds.

\begin{Theo}
\label{Lazas}
Let $\mathcal{S}$ be a locally free sheaf on a smooth projective curve $Y$ over an algebraically closed field $k$. Fix a cohomology class $c \in H^1(Y, \mathcal{S})$ which defines an extension $\mathcal{S}'$ of $\mathcal{O}_Y$ by $\mathcal{S}$. Then the following are equivalent:
\begin{enumerate}[(i)]
 \item{The cohomology class $c$ is almost zero.}
\item{For every $\eps >0$ there is a finite $k$-morphism $\varphi: Y' \to Y$ of smooth projective curves, a line bundle $\mathcal{L}$ on $Y'$ such that $\deg \mathcal{L}/\deg \varphi < \eps$ and such that there exists a surjection $\varphi^\ast \mathcal{S}'^\vee \to \mathcal{L}$ which does not factor through $\varphi^\ast \mathcal{S}^\vee$.}
\item{For every $\eps >0$ there is a curve $C$ in $\mathbb{P}(\mathcal{S}'^\vee)$ not contained in $\mathbb{P}(\mathcal{S}^\vee)$ such that $\varphi: C \to Y$ induced by the projection is finite and such that \[\frac{C.\mathbb{P}(\mathcal{S}^\vee)}{\deg \varphi} < \eps.\]} 
\end{enumerate}
\end{Theo}
\begin{proof}
We prove the implication from (i) to (ii).
For $\eps > 0$ let $\varphi:Y' \to Y$ be a finite morphism of smooth projective curves and $\mathcal{L}$ a line bundle on $Y'$ such that $\deg \mathcal{L}/\deg \varphi < \eps$ and $s$ a nonzero global section of $\mathcal{L}$ with $s \varphi^\ast(c) = 0$. We consider the exact sequence $0 \to \varphi^\ast\mathcal{S} \to \varphi^\ast\mathcal{S}' \to \mathcal{O}_{Y'} \to 0$, tensor with $\mathcal{L}$ and take cohomology. We thus obtain an exact sequence
\[0 \xrightarrow{} H^0(Y',\varphi^\ast\mathcal{S} \otimes \mathcal{L})\xrightarrow{} H^0(Y',\varphi^\ast\mathcal{S}' \otimes \mathcal{L}) \xrightarrow{} H^0(Y', \mathcal{L})\xrightarrow{\delta} H^1(Y', \varphi^\ast\mathcal{S} \otimes \mathcal{L}).
 \]
Now, since $\delta(s) = s\varphi^\ast(c) = 0$ we have that $s$ stems from a global section in $\varphi^\ast \mathcal{S}' \otimes \mathcal{L}$. This defines a morphism $0 \to \mathcal{L}^\vee \to \varphi^\ast \mathcal{S}'$ which does not factor through $\varphi^\ast \mathcal{S}$ since $s$ is nonzero. Passing to the saturation $\mathcal{L}'^\vee$ of $\mathcal{L}^\vee$ in $\varphi^\ast \mathcal{S}'$ yields a morphism $\varphi^\ast \mathcal{S}'^\vee \to \mathcal{L}' \to 0$ with $\frac{\deg\mathcal{L}}{\deg \varphi} < \eps$ which a fortiori does not factor through $\varphi^\ast \mathcal{S}^\vee$.

For the implication from (ii) to (i) let $\eps >0$. Assume that $\mathcal{L}$ is a line bundle on a smooth curve $\varphi: Y' \to Y$ finite over $Y$ such that $\deg \mathcal{L}/\deg \varphi < \eps$. Furthermore, assume that $\varphi^\ast \mathcal{S}'^\vee \to \mathcal{L}$ is a surjection which does not factor through $\varphi^\ast \mathcal{S}^\vee$. Therefore, we have a nonzero section $s$ in $H^0(Y', \varphi^\ast \mathcal{S}' \otimes \mathcal{L})$ associated to $\varphi^\ast \mathcal{S}'^\vee \to \mathcal{L} \to 0$. Since this morphism does not factor through $\varphi^\ast \mathcal{S}^\vee$ the section $s$ does not stem from $H^0(Y', \varphi^\ast \mathcal{S} \otimes \mathcal{L})$. Again we have an exact sequence \[0 \xrightarrow{} H^0(Y',\varphi^\ast\mathcal{S} \otimes \mathcal{L})\xrightarrow{} H^0(Y',\varphi^\ast\mathcal{S}' \otimes \mathcal{L}) \xrightarrow{\sigma} H^0(Y', \mathcal{L})\xrightarrow{\delta} H^1(Y', \varphi^\ast\mathcal{S} \otimes \mathcal{L}),
 \]
where $\sigma(s) \neq 0$ and $\sigma(s) \cdot \varphi^\ast(c) = \delta (\sigma (s)) = 0$ as desired.

The equivalence of (ii) and (iii) is given by the correspondence described in \cite[Proposition II.7.12]{hartshornealgebraic}. A surjective morphism $\varphi^\ast \mathcal{S}'^\vee \to \mathcal{L}$ corresponds to a morphism $Y' \to \mathbb{P}(\mathcal{S}'^\vee)$ over $Y$ and that it does not factor through $\mathcal{S}^\vee$ means precisely that its image is not contained in $\mathbb{P}(\mathcal{S}^\vee$).
For the rest of the claim assume that $C$ is finite over $Y$. The inclusion $C \to \mathbb{P}(\mathcal{S}'^\vee)$ corresponds to a line bundle $\mathcal{L}$ on $C$. Pulling back to the normalisation $Y'$ of $C$ yields (ii) since $\deg \mathcal{L} = C.\mathbb{P}(\mathcal{S}^\vee)$. 
Conversely, assuming (ii) the image of $Y'$ in $\mathbb{P}(\mathcal{S}'^\vee)$ yields $C$.  We refer to the proof of \cite[Theorem 2.3]{brennerslope} for a detailed discussion. 
\end{proof}

\begin{Bem}
\label{BemSupportCond}
The heuristic behind the condition that the curve be not contained in $\mathbb{P}(\mathcal{S}^\vee)$ is that otherwise we could only construct a syzygy whose first term is $0$. Hence, we would obtain a relation $(0, a_1, \ldots, a_n)$ which is not interesting with respect to dagger closure. We will refer to this as the \emph{support condition} in the following.
\end{Bem}

\begin{Bem}
Let $Y$ be a smooth projective curve over an algebraically closed field $k$ and let $0 \to \mathcal{S} \to \mathcal{S}' \to \mathcal{O}_Y \to 0$ be an exact sequence of locally free sheaves. Then there are cases where $\mathcal{S}'^\vee$ is not ample but all the curves that contradict the ampleness in the sense of Seshadri's Theorem (\cite[Theorem I.7.1]{hartshorneamplesubvarieties}) are contained in the support of $\mathbb{P}(\mathcal{S}^\vee)$. Indeed, by \cite[Theorem 2.3]{brennertightplus} and Theorem \ref{HNFaztight} below we find curves contradicting the ampleness that do not lie in the support of $\mathbb{P}(\mathcal{S}^\vee)$ if and only if the complement $\mathbb{P}(\mathcal{S}'^\vee) \setminus \mathbb{P}(\mathcal{S}^\vee)$ is not affine. But $\mathcal{S}'^\vee$ may not be ample and still have affine complement $\mathbb{P}(\mathcal{S}'^\vee) \setminus \mathbb{P}(\mathcal{S}^\vee)$. 

Specifically, let $Y = \mathbb{P}^1_k =\Proj k[x,y]$, $\mathcal{S} = \Syz(x^4, y^4, x^4)(2) = \mathcal{O}_Y \oplus \mathcal{O}_Y(-2)$ and $c = \delta(xy) \in H^1(Y, \mathcal{S})$ so that $\mathcal{S}' = \Syz(x^4, y^4, x^4, xy)$ (cf.\ \cite[Example 7.3]{brennerbarcelona}). The forcing divisor is then not ample since $\mathcal{S}'^\vee$ surjects onto $\mathcal{S}^\vee$, hence has the non-ample quotient $\mathcal{O}_Y$.  But $\mathbb{P}(\mathcal{S}'^\vee) \setminus \mathbb{P}(\mathcal{S}^\vee)$ is affine since $xy \notin (x^4, y^4) =  (x^4, y^4)^{\star}$ as $k[x,y]$ is regular of dimension $2$.   
\end{Bem}

\section{Almost zero and graded dagger closure}

Our next goal is to prove that the notion of almost zero as in Definition \ref{almostzero} characterises dagger closure. The next two lemmata are well-known but the authors are not aware of a suitable reference. Hence, we shall provide proofs.

\begin{Le}
\label{RelDifferentiale}
Let $A$ be a noetherian ring containing a field $k$ of characteristic $p \geq 0$, $a \in A^\times$ and $B = A[T]/(T^n -a )$. If $(\chara k, n) =1 $ then the relative differentials $\Omega_{B/A}$ vanish.
\end{Le}
\begin{proof}
To begin with, $\Omega_{B/A}$ is generated by $dT$ so that it is enough to show that $dT = 0$. We have $dT^n = da =0$ and $dT^n = n T^{n-1} dT$. Hence, $T^{n-1} dT = 0$ and since $T^n$ is a unit, so is $T^{n-1}$. Thus $dT = 0$ as desired.
\end{proof}

\begin{Le}
\label{Ltortriv}
Let $Y$ be a smooth projective curve over an algebraically closed field $k$. Denote by $\mathcal{T}$ a torsion element of $\Pic Y$. Then there is a finite morphism $\varphi:Y' \to Y$ such that $\varphi^\ast \mathcal{T} \cong \mathcal{O}_{Y'}$ and such that $Y'$ is smooth.
\end{Le}
\begin{proof}
See also \cite[Ex.\ IV.2.7]{hartshornealgebraic}.
Denote the order of $\mathcal{T}$ by $n$ and consider the coherent $\mathcal{O}_Y$-algebra $\mathcal{A} = \bigoplus_{i=0}^{n-1} \mathcal{T}^i$, where the multiplication is induced by a fixed isomorphism $\alpha: \mathcal{T}^n \to \mathcal{O}_Y$. First, we show that $\psi: \Spec \mathcal{A} \to Y$ is finite. The morphism is affine by \cite[Ex.\ II.5.17]{hartshornealgebraic}. Take an open affine cover $U_j$ of $Y$ such that the $\mathcal{T}\vert_{U_j}$ are free. It follows that the $\mathcal{T}^i\vert_{U_j}$ are also free and $\mathcal{O}_{\psi^{-1}(U_j)} = \mathcal{A}(U_j)$ is a free $\mathcal{O}_{U_j}$-module of rank $n$. 
Next, we show that $\psi^{\ast} \mathcal{T} = \mathcal{O}_{Y'}$. One has $(\psi^{\ast} \mathcal{T})_U = \mathcal{A}_U \otimes_{\mathcal{O}_U} \mathcal{T} = \mathcal{A}_U$ and clearly these isomorphisms glue.

Note that if $(\chara k, n) = 1$ then $\Spec \mathcal{A} = X$ is \'etale over $Y$.
 Indeed, locally over $U = U_j$ we have that $\mathcal{T}_U$ is generated by some element $t$. We therefore may identify $\mathcal{O}_{\psi^{-1}(U)}$ with $\mathcal{O}_{U}[T]/(T^n -\alpha(t^n))$. Applying \cite[Ex.\ III.10.3]{hartshornealgebraic} it is enough to observe that $\Omega_{X/Y} = 0$ by Lemma \ref{RelDifferentiale}, that the extension over the stalks is separable since $(\chara k, n) =1$ and that the morphism is flat by \cite[Proposition 4.1 (b)]{Eisenbud}. Pulling back to an irreducible component $Y'$ which dominates $Y$ (in fact, every component dominates $Y$) we have the desired morphism. 

If the characteristic of the base divides $n$ we can write $n = p^e m$ with $p \nmid m$. Then we consider a Frobenius pullback $F^e: Y' \to Y$ and thus have $(F^e)^{\ast}\mathcal{T} = \mathcal{T}^{p^e}$. Note that $Y'$ is smooth and that $\mathcal{T}^{p^e}$ is an $m$-torsion element so that we may apply the previous argument.
\end{proof}

\begin{Ko}
\label{globallygeneratedafterpullback}
Let $Y$ be a smooth projective curve over an algebraically closed field and let $\mathcal{L}$ be an ample line bundle on $Y$. Then there is a finite morphism $\varphi: Y' \to Y$ of smooth projective curves such that $\varphi^\ast \mathcal{L}$ is generated by global sections.
\end{Ko}
\begin{proof}
Applying Lemma \ref{allgPullback} to $\mathcal{L}$ with $\mathcal{O}(1) = \mathcal{L}$ (and $N =m$) we have that after a finite pullback $\psi: X \to Y$ there is a line bundle $\mathcal{M}$ on $X$ which is generated by global sections such that $\psi^\ast \mathcal{L}^m = \mathcal{M}^m$. Hence, $\psi^\ast \mathcal{L} \otimes \mathcal{T} = \mathcal{M}$ where $\mathcal{T}$ is an $m$-torsion element of the Picard group. Applying Lemma \ref{Ltortriv} we are done.
\end{proof}

\begin{Le}
\label{Pullbacksupport}
Let $\varphi: Y' \to Y$ be a finite morphism of smooth projective curves over an algebraically closed field, $\mathcal{S}$ a locally free sheaf on $Y$ and $\mathcal{S}'$ an extension corresponding to $c \in H^1(Y,\mathcal{S})$. If $\mathcal{L}$ is a line bundle on $Y$ and $\mathcal{S}'^\vee \to \mathcal{L} \to 0$ a surjection which does not factor through $\mathcal{S}^\vee$, then $\varphi^\ast \mathcal{S}'^\vee \to \varphi^\ast \mathcal{L} \to 0$ does not factor through $\varphi^\ast \mathcal{S}^\vee$.
\end{Le}
\begin{proof}
Note that we have an exact sequence $0 \to \mathcal{S} \to \mathcal{S}' \to \mathcal{O}_Y \to 0$ and hence an inclusion $H^0(Y, \mathcal{S}) \subseteq H^0(Y, \mathcal{S}')$.
The exact sequence $\mathcal{S}'^\vee \to \mathcal{L} \to 0$ yields after dualising and tensoring with $\mathcal{L}$ a global section $s \in H^0(Y, \mathcal{S}' \otimes \mathcal{L})$ which is not contained in $H^0(Y, \mathcal{S} \otimes \mathcal{L})$.
We have a commutative diagram \[
\begin{xy} \xymatrix{0 \ar[r] & H^0(Y, \mathcal{S} \otimes \mathcal{L}) \ar[r] \ar[d] & H^0(Y, \mathcal{S' \otimes \mathcal{L}}) \ar[d] \ar[r] & H^0(Y, \mathcal{L}) \ar[d]  \\ 0 \ar[r]  & H^0(Y', \varphi^\ast(\mathcal{S} \otimes \mathcal{L})) \ar[r] & H^0(Y', \varphi^\ast(\mathcal{S' \otimes \mathcal{L}})) \ar[r] & H^0(Y', \varphi^\ast \mathcal{L}) }\end{xy}\]
where the vertical arrows are injective (locally this is just a base change $R^n \to R^n \otimes_R S$ whence the injectivity). 
Since the rightmost vertical arrow is injective $\varphi^\ast(s)$ cannot lie in $H^0(Y', \varphi^\ast( \mathcal{S}\otimes \mathcal{L}))$. And this means that the surjection still does not factor.
\end{proof}

For future reference we fix the following

\begin{Sit}
\label{Sit1}
Let $R$ be a normal standard graded two-dimensional integral $k$-algebra, where $k$ is an algebraically closed field of arbitrary characteristic. Moreover, let $I = (f_1, \ldots, f_n)$ be an $R_+$-primary homogeneous ideal, where the $f_i$ are homogeneous of degrees $d_i$. Write $Y = \Proj R$ and fix a homogeneous element $f_0$ in $R$ of degree $d_0$. Write $\mathcal{S} = \Syz(f_1, \ldots, f_n)(d_0)$.
\end{Sit}

\begin{Theo}
\label{almostzerodagger}
Assume Situation \ref{Sit1}. Then $c = \delta(f_0) \in H^1(Y,\mathcal{S})$ is almost zero if and only if $f_0$ is contained in $I^{\dagger \text{GR}}$.
\end{Theo}

\begin{proof}
Assume that $f_0 \in I^{\dagger \text{GR}}$. Thus by Lemma \ref{Ldeg}, we find for $m \in \mathbb{N}$ a nonzero element $a_m$ of degree $\frac{1}{m}$ such that $a_m f_0 \in IS$, where $S$ is a finite $\mathbb{Q}$-graded extension domain of $R$. Note that we may assume that the minimal degree of $S$ is $\frac{1}{m}$ -- writing $a_m f_0 = \sum_i a_{im} f_i$ one necessarily has $d_0 \geq \min\{d_1, \ldots, d_n\}$. Hence, all relevant data are contained in a finite ring extension $R[a_m, a_{1m}, \ldots, a_{nm}] \subseteq R^{+ \text{GR}}$ whose minimal degree is $\frac{1}{m}$. Furthermore, we may assume that $S$ is normal and that $Y' = \Proj S$ is covered by finitely many standard open sets coming from elements of degree $1$ after regrading. Moreover, we may assume that elements of degree $1$ in $R$ are of degree $m$ in the regraded copy of $S$. To accomplish this adjoin roots to $S$ so that $S$ is generated in degree $\frac{1}{m}$ (Lemma \ref{stdgraded}) then normalise and regrade by multiplying the grading by $m$ so that $S$ is $\mathbb{N}$-graded. Thus we have all the desired properties and $S(1)^\sim$ is invertible.

Pulling back the whole situation to $Y' = \Proj S$ we therefore have that $a_m \delta(f_0) = 0$ since $a_m f_0$ is contained in the extended ideal. The multiplication by $a_m$ is induced by a morphism $\mathcal{O}_{Y'} \to \mathcal{O}_{Y'}(1)$. Now one can argue as in the proof of Proposition \ref{RSSazimpliesours}. Indeed, tensoring $\mathcal{O}_{Y'} \to \mathcal{O}_{Y'}(1)$ with $\mathcal{S}$ and taking cohomology yields that $H^1(Y', \varphi^\ast \mathcal{S}) \to H^1(Y', \varphi^\ast \mathcal{S} \otimes \mathcal{O}_{Y'}(1))$ maps $\delta(f_0)$ to $a_m \delta(f_0) = 0$, where $\varphi: Y' \to Y$ is induced by the inclusion of rings. Moreover, $\frac{\deg \mathcal{O}_{Y'}(1)}{\deg \varphi} = \frac{1}{m} \deg \mathcal{O}_Y(1)$.

As to the other implication, Theorem \ref{Lazas} implies that for $\eps > 0$ there is a line bundle $\mathcal{L}$ on some smooth curve $f: Y' \to Y$ finite over $Y$ such that $\frac{\deg \mathcal{L}}{\deg f} < \eps$ and an exact sequence $f^\ast \mathcal{S}'^\vee \to \mathcal{L} \to 0$ which does not factor through $f^\ast\mathcal{S}^\vee$.

Consider the ring $S = \bigoplus_{n \geq 0} \Gamma(Y', f^\ast \mathcal{O}_Y(n))$ -- this is a finite $\mathbb{N}$-graded extension of $R$ by Proposition \ref{Leholger}. Moreover, since $\mathcal{O}_Y(1)$ is generated by global sections (it is even very ample) the line bundle $f^\ast \mathcal{O}_Y(1)$ is also generated by global sections. Hence, Proposition \ref{PCover} implies that $\Proj S$ is covered by finitely many standard open sets coming from elements $s_1, \ldots, s_l$ of degree $1$ in $S$. This in turn allows us to apply Lemma \ref{Pullback} and to adjoin $m$th roots of the $s_1, \ldots, s_l$ to obtain rings $S'_m$ that are still covered in degree $1$ after regrading. Varying $m$ we obtain finite morphisms $g_m: Y''_m \to Y'$, where $Y''_m = \Proj S'_m$ so that $\deg \mathcal{O}_{Y''_m}(1)/ \deg (fg_m)^\ast \mathcal{O}_Y(1) = \frac{1}{m}$.

In order to obtain syzygies we fix an $m \gg 0$ and omit the index. We have that $\deg g^\ast \mathcal{L}^\vee/ \deg (fg)$ is rational and independent of $g$. Hence, choosing $m$ sufficiently large we can find a positive integer $t-1$ such that $\frac{t-1}{m} \deg \mathcal{O}_Y(1) = \deg g^\ast \mathcal{L}/ \deg (fg)$. Possibly choosing a larger $m$ (and then a different $t$) we may moreover assume that $\frac{t}{m} \deg \mathcal{O}_Y(1) < \eps$.
Now consider $g^\ast \mathcal{L}^\vee$ and twist by $\mathcal{O}_{Y''}(t)$ for a $t$ as above so that \[0 < \frac{\deg g^\ast \mathcal{L}^\vee(t)}{\deg (fg)} < \eps.\] Applying Corollary \ref{globallygeneratedafterpullback} we may assume that $g^\ast \mathcal{L}^\vee(t)$ is generated by global sections after a finite pullback and pulling back again if necessary we may assume that $Y''$ is smooth. By Proposition \ref{Propringext} we have a finite extension of graded domains $R \subseteq \bigoplus_{n \geq 0} \Gamma(Y'', \mathcal{O}_{Y''}(n)) = T$ and an exact sequence \[ \begin{xy} \xymatrix{0 \ar[r] & g^\ast \mathcal{L}^{\vee}(t) \ar[r] & (fg)^\ast\mathcal{S}'(t).}
\end{xy}
\]

Recalling the presenting sequence for the syzygy bundle we see that global sections of $g^\ast \mathcal{L}^\vee(t)$ define syzygies of total degree $m d_0 + t$ on $\Proj T$. We therefore have syzygies $(a_0, \ldots, a_n)$ corresponding to relations $-a_0 f_0 = \sum_i a_i f_i$, where $\deg a_0 = \frac{t}{m} < \eps$ and we still need to make sure that we find a syzygy with $a_0 \neq 0$. But otherwise $g^\ast \mathcal{L}^\vee(t)$ would factor through $((fg)^\ast \Syz(f_1, \ldots, f_n)(d_0))(t)$ since it is generated by global sections. Twisting by $\mathcal{O}_{Y''}(-t)$ and applying Lemma \ref{Pullbacksupport} would imply that $\mathcal{L}$ factored as well contradicting our assumptions.
\end{proof}



\begin{Prop}
\label{PConversealmostzero}
Let $Y$ be a smooth projective curve over an algebraically closed field. Fix an ample line bundle $\mathcal{O}_Y(1)$ and let $\mathcal{S}$ be a locally free sheaf of rank $r$ such that $\det \mathcal{S} = \mathcal{O}_Y(n)$ for some integer $n$. If $\mathcal{S}$ is almost zero then one can choose the annihilating line bundles as suitable roots of $\mathcal{O}_Y(1)$.
\end{Prop}
\begin{proof}
By Proposition \ref{Syzdet}, $\mathcal{S}$ is up to twist a syzygy bundle. This means that we have an exact sequence\footnote{A priori the $d$'s might be different but the proof of Proposition \ref{Syzdet} actually shows that they are equal.} $0 \to \mathcal{S} \to \bigoplus_{i=0}^r \mathcal{O}_Y(d) \to \mathcal{O}_Y(l) \to 0$. Making a finite pullback we may assume that $\mathcal{O}_Y(1)$ is generated by global sections. Hence, by Proposition \ref{PCover} its section ring is covered by finitely many standard open sets coming from elements of degree $1$.  Arguing as in the proof of Theorem \ref{almostzerodagger} we see that we can annihilate $\mathcal{S}$ by roots of $\mathcal{O}_Y(1)$.
\end{proof}

\begin{Prop}
\label{PTorsionbeliebigannulierbar}
Let $Y$ be a smooth projective curve over an algebraically closed field and $\mathcal{S}$ a vector bundle such that $\deg (\det \mathcal{S}) = 0$. Then if $\mathcal{S}$ is almost zero the annihilators can be chosen as sections of roots of any ample line bundle $\mathcal{L}$ on $Y$.
\end{Prop}
\begin{proof}
If $\det \mathcal{S}$ is torsion then the determinant is trivial after a finite pullback and the pullback of $\mathcal{S}$ is a twisted syzygy bundle with respect to any embedding. Thus in this case the claim follows from Proposition \ref{PConversealmostzero} above.

So assume that $\det \mathcal{S}$ is not a torsion element and fix an ample line bundle $\mathcal{L}$ on $Y$. In light of Corollary \ref{globallygeneratedafterpullback} and Lemma \ref{allgPullback} there is a finite morphism $\varphi: Y' \to Y$ of smooth projective curves such that there is a root $\mathcal{L}'$ on $Y'$ of $\mathcal{L}$ which is generated by global sections and such that $\deg \mathcal{L}'/\deg \varphi < \frac{\eps}{2}$. We note that pullbacks and $\det$ commute. Consider $\mathcal{L}'' = \mathcal{L}' \otimes  (\det \varphi^\ast \mathcal{S})^\vee$. After a finite pullback we may assume that there is a $\rk \mathcal{S}$th root $\mathcal{M}$ of $\mathcal{L}''$ which is globally generated.

In particular, there is a global section $s \neq 0$ of $\mathcal{M}$ which induces a morphism $H^1(Y', \varphi^\ast \mathcal{S}) \to H^1(Y',\varphi^\ast \mathcal{S} \otimes \mathcal{M})$. Moreover, $\det (\varphi^\ast \mathcal{S} \otimes \mathcal{M})$ = $\det \varphi^\ast \mathcal{S} \otimes \mathcal{M}^{\rk \mathcal{S}} = \mathcal{L}'$. Hence, $\varphi^\ast \mathcal{S} \otimes \mathcal{M}$ is a twisted syzygy bundle with respect to $\mathcal{L}'$. Since the latter is a root of $\mathcal{L}$ the claim follows from Proposition \ref{PConversealmostzero} above.
\end{proof}

\begin{Bem}
\label{AllesfixAnnulierbar}
Let $\mathcal{S}$ be a locally free sheaf on a smooth projective curve $Y$ over an algebraically closed field. Then $\mathcal{S}$ is almost zero if and only if it can be annihilated with respect to a suitable fixed embedding. That is, the annihilators can be chosen as roots of a suitable fixed line bundle $\mathcal{L}$ -- this follows from Propositions \ref{PConversealmostzero} and \ref{PTorsionbeliebigannulierbar} above. In particular, our definition of almost zero is equivalent to Definition \ref{RSSalmostzero} whenever both are applicable.
\end{Bem}

\section{Almost zero for line bundles}
In this section we will investigate line bundles with respect to the property of being almost zero. From the viewpoint of dagger closure this corresponds to the parameter case, i.\,e.\ one has two ideal generators $f_1, f_2$. In characteristic zero this notion will only depend on the degree of the line bundle for a given curve -- see Theorem \ref{Plbaz}. In positive characteristic the situation will be a little bit more subtle.

\begin{Le}
Let $Y$ be a smooth projective curve over an algebraically closed field and $\mathcal{E}$ a locally free sheaf. Let $Y'$ be an irreducible curve in $\mathbb{P}(\mathcal{E})$ which dominates the base. Write $a \xi^{n-1} + b \xi^{n-2}f$ for its numerical equivalence class, where $f$ is a fibre of $\pi: \mathbb{P}(\mathcal{E}) \to Y$ and $\xi$ a Weil divisor whose linear equivalence class corresponds to $\mathcal{O}_{\mathbb{P}(\mathcal{E})}(1)$. Let $\varphi$ be the morphism $Y' \to Y$ induced by $\pi$. Then $a = \deg(\varphi)$.
\end{Le}
\begin{proof}
We have the following commutative diagram
\[
\begin{xy}
 \xymatrix{& \mathbb{P}(\mathcal{E}) \ar[d]^{\pi} \\
Y' \ar[r]^\varphi \ar[ur]^i & Y }
\end{xy}
\] and replacing $Y'$ by its normalisation we may assume that $Y'$ is smooth (but $i$ will no longer be a closed immersion). Intersecting $i_\ast Y'$ with a fibre we have $a = i_\ast Y'.f$ which is equal to $i_\ast(Y'.i^\ast f)$ by the projection formula. Thus we have $i_\ast Y'.f = i_\ast(Y'. i^\ast \pi^\ast Q)$, where $Q$ is a point on $Y$. Considering this as an element of $\mathbb{Z}$ (via the degree map) rather than of the Chow ring we may omit the pushforward. Thus the latter is equal to $Y'. \varphi^\ast Q$ which is equal to $\deg \varphi$ by \cite[Proposition II.6.9]{hartshornealgebraic}.
\end{proof}

\begin{Prop}
\label{paramaz}
Let $Y$ be a smooth projective curve over an algebraically closed field. If $\mathcal{L}$ is a line bundle of degree $\geq 0$ then $\mathcal{L}$ is almost zero.
\end{Prop}
\begin{proof}
Fix $c \in H^1(Y, \mathcal{L})$. This defines an extension $0 \to \mathcal{L} \to \mathcal{E} \to \mathcal{O} \to 0$. Dualising we have that $X = \mathbb{P}(\mathcal{E}^\vee)$ is a normalised ruled surface with $e$-invariant $e = \deg \mathcal{L} \geq 0$. This is normalised since if we tensor the dualised short exact sequence with a line bundle $\mathcal{M}$ of negative degree then $H^0(Y, \mathcal{E}^\vee \otimes \mathcal{M}) = 0$. It also follows that $\mathcal{E}^\vee$ is not ample, since it surjects onto $\mathcal{L}^\vee$ which has degree $\leq 0$. Fix a section $Y_0$ that is equivalent to $\mathcal{O}_{\mathbb{P}(\mathcal{E}^\vee)}(1)$.

In order to apply Theorem \ref{Lazas} we need to find for every $\eps > 0$ a curve $C$ in $X$ such that $\varphi: C \to Y$ is dominant and $0 < Y_0.C/\deg \varphi < \eps$. The curve $C$ then defines a section of $\mathbb{P}(\varphi^\ast \mathcal{E}^\vee)$ which corresponds to a line bundle $\mathcal{M}$ on $C$ with $\deg \mathcal{M} = Y_0.C$ such that $\varphi^\ast \mathcal{E}^\vee \to \mathcal{M} \to 0$ is exact (see \cite[Proposition 7.12]{hartshornealgebraic}). Also note that since $\deg \varphi^\ast \mathcal{L}^\vee \leq 0$ we cannot have a surjection $\varphi^\ast \mathcal{L}^\vee \to \mathcal{M} \to 0$. 

By \cite[Proposition V.2.20 (a)]{hartshornealgebraic} an irreducible curve $C' \neq Y_0$ is numerically equivalent to $a Y_0 + bf$ where $f$ is the class of a fibre and $a >0, b \geq ae$. Moreover, $C'$ is ample if in addition $b > ae$ (loc.\,cit.). So fix a divisor $C'$ numerically equivalent to $a Y_0 + bf$ with $a > 0$ and choose $b = ae +1$. Then $nC' = na Y_0 + n(ae +1 ) f$ is very ample for $n \gg 0$. Use $nC'$ to embed $X$ into some $\mathbb{P}^N$ and apply Bertini's Theorem \cite[Theorem II.8.18 and Remark III.7.9.1]{hartshornealgebraic} to see that there is an irreducible nonsingular curve $C$ on $X$ which is linearly equivalent to $nC'$. 
Since $C$ is not contained in a fibre it dominates the base $Y$ and we have an induced morphism $\varphi: C \to Y$, where $\deg \varphi = na$. This yields $\frac{C.Y_0}{\deg \varphi} = \frac{n}{an}= \frac{1}{a}$. Therefore, choosing $a \gg 0$ completes the proof.
\end{proof}

One can also prove Proposition \ref{paramaz} in characteristic zero using \cite[Corollary 3.5]{robertssinghannihilators}.
Furthermore, one can prove this in characteristic $p > 0$ using the $k$-linear Frobenius.

\begin{Bem}
We remark that one cannot settle the support condition (cf. Remark \ref{BemSupportCond}) by just looking at numerical equivalence classes if $n \geq 3$.
To see this assume Situation \ref{Sit1} and assume that $\mathcal{S}$ is strongly semistable. Write $\mathcal{S}' = \Syz(f_0, \ldots, f_n)(d_0)$, where $f_0$ is a homogeneous element of degree $d_0$ and assume that $\deg \mathcal{S} > 0$. The dual of $\mathcal{S}'$ is then given by the exact sequence \[\begin{xy}\xymatrix{0 \ar[r]& \mathcal{O}_Y \ar[r]& \mathcal{S}'^\vee \ar[r]& \mathcal{S}^\vee  \ar[r]& 0.}\end{xy}\] Since both $\mathcal{S}^\vee$ and $\mathcal{O}_Y$ are strongly semistable and since $\deg \mathcal{S}^\vee < 0$, the inclusion $\mathcal{O}_Y \subset \mathcal{S}'^\vee$ is a strong Harder-Narasimhan filtration of $\mathcal{S}'^\vee$. 

Now, in \cite[Lemma 2.3]{fulgereffcycles} Fulger shows that one has an isomorphism of the closure of the effective cones of $1$-cycles of $\mathbb{P}(\mathcal{S}^\vee)$ and $\mathbb{P}(\mathcal{S}'^\vee)$. Hence, one cannot discern by numerical properties whether a given curve lying on $\mathbb{P}(\mathcal{S}'^\vee)$ is already contained in $\mathbb{P}(\mathcal{S}^\vee)$.
Note that Fulger proves this only in characteristic zero but this result is probably true more generally in arbitrary characteristic if one looks at a strong Harder-Narasimhan filtration.

If $\deg \mathcal{S}= 0$ then $\mathcal{S}'^\vee$ itself is strongly semistable. 
By \cite[Lemma 2.2]{fulgereffcycles}, which also holds in characteristic $p > 0$ if one assumes strong semistability, the closure of the effective cone of $1$-cycles is spanned by $\xi'^{n-1}, \xi'^{n-2}f'$, where $f'$ is the class of a fibre and $\xi'$ is the class of $\mathcal{O}_{\mathbb{P}(\mathcal{S}'^\vee)}(1)$. Furthermore, we have a closed immersion $i: \mathbb{P}(\mathcal{S}^\vee) \to \mathbb{P}(\mathcal{S}'^\vee)$ and denoting $\mathcal{O}_{\mathbb{P}(\mathcal{S}^\vee)}(1)$ by $\xi$ one has that $i^\ast \xi' = \xi$ and $i^\ast f' = f$, where $f$ denotes a fibre in $\mathbb{P}(\mathcal{S}^\vee)$. We want to show that $i_\ast \xi^{n-2}$ and $i_\ast (\xi^{n-3}f)$ are numerically equivalent to $\xi'^{n-1}$ and $\xi'^{n-2}f'$ respectively. So let $D = a \xi' + bf'$ be a divisor in $\mathbb{P}(\mathcal{S}'^\vee)$. Then we have $i_\ast \xi^{n-2}.D = i_\ast(\xi^{n-2}.i^\ast D)$ by the projection formula and this is equal to $b + a \xi'^n = \xi'^{n-1}.D$. Similarly $i_\ast \xi^{n-3}f.D = a = \xi'^{n-2}f.D$.
\end{Bem}

\begin{Bsp}
\label{ProjMultiplikationsmorphismen}
We now want to provide an example of an extension $0 \to \mathcal{S} \to \mathcal{S}' \to \mathcal{O} \to 0$ where we have curves with positive intersection lying in the support of $\mathbb{P}(\mathcal{S}^\vee)$ that contradict the ampleness of $\mathcal{S}'^\vee$.

Let $k$ be an algebraically closed field and let $Y$ be an elliptic curve over $k$ (e.\,g.\ $Y = k[x,y,z]/(x^3 + y^3 + z^3)$). Fix any morphism $\varphi: Y \to \mathbb{P}^1_k$ of degree $2$ and write $\mathcal{O}_Y(1) = \varphi^\ast \mathcal{O}_{\mathbb{P}^1_k}(1)$. Consider now $\mathcal{S} = \mathcal{O}_Y \oplus \mathcal{O}_Y$ and the short exact sequence $0 \to \mathcal{O}_Y(-1) \to \mathcal{S} \to \mathcal{O}_Y(1) \to 0$ (this is just the pullback of a twist of the defining sequence of the cotangent bundle on $\mathbb{P}^1_k$ -- cf.\ \cite[Theorem II.8.13]{hartshornealgebraic}).

We have $H^1(Y, \mathcal{S}) = 2$ and dualising this sequence we obtain a surjection $\mathcal{S}^\vee \to \mathcal{O}_Y(1)$. Pulling back along multiplication morphisms $N_Y: Y' \to Y$ (for some $N \in \mathbb{N}$) we obtain surjections $N_Y^\ast \mathcal{S}^\vee \to \mathcal{O}_{Y'}(1)$ so that $\frac{\deg \mathcal{O}_{Y'}(1)}{\deg N_Y} = \frac{\deg \mathcal{O}_{Y}(1)}{N^2}$. Therefore, we have curves with positive intersection lying in the support of $\mathbb{P}(\mathcal{S}^\vee)$ which contradict the ampleness of $\mathcal{S}'^\vee$.

Note that $\mathcal{S}$ is actually a twisted syzygy bundle by virtue of Proposition \ref{Syzdet}. We have presenting sequences \[\begin{xy}\xymatrix{0 \ar[r]& \mathcal{S} \ar[r]& \mathcal{O}_Y(m)^3 \ar[r]& \mathcal{O}_Y(3m) \ar[r]& 0.}\end{xy} \]Choosing $m = 1$ we have the surjective connecting homomorphism \[\begin{xy}\xymatrix{\delta_\mathcal{S}: H^0(Y, \mathcal{O}_Y(3)) \ar[r]& H^1(Y, \mathcal{S}).}\end{xy}\]Hence, every class is realised via an element that stems from the section ring induced by $\mathcal{O}_Y(1)$.
\end{Bsp}

The following lemma is false in positive characteristic as we will see in Remark \ref{BemAzLinebundlesPos} below.

\begin{Le}
\label{Lnotamplenaz}
Let $Y$ be a smooth projective curve over an algebraically closed field of characteristic zero. Let $\mathcal{L}$ be a line bundle of negative degree. Then any nonzero $c \in H^1(Y, \mathcal{L})$ is not almost zero. 
\end{Le}
\begin{proof}
The class $c$ defines a non-trivial extension $0 \to \mathcal{O}_Y \to \mathcal{E} \to \mathcal{L}^\vee \to 0$ and $\mathcal{L}^\vee$ is ample since $\deg \mathcal{L}^\vee > 0$. By \cite[Proposition 2.2]{giesekerample} every quotient bundle of $\mathcal{E}$ has positive degree, that is, $\mu_{\min}(\mathcal{E}) > 0$. Since the characteristic is zero this implies that $\mathcal{E}$ is ample by \cite[Theorem 2.3]{brennerslope} and we are done by Theorem \ref{Lazas}.
\end{proof}

\begin{Theo}
\label{Plbaz}
Let $Y$ be a smooth projective curve over an algebraically closed field of characteristic zero. A line bundle $\mathcal{L}$ on $Y$ is almost zero if and only if $\deg \mathcal{L} \geq 0$ or if $Y = \mathbb{P}^1$ and $\mathcal{L} = \mathcal{O}_{\mathbb{P}^1}(-1)$.
\end{Theo}
\begin{proof}
For the only if part we assume that $\deg \mathcal{L} < 0$ and apply Lemma \ref{Lnotamplenaz} to see that $c \in H^1(Y, \mathcal{L})$ is almost zero if and only if $c = 0$. Thus $\mathcal{L}$ is almost zero if and only if $H^1(Y, \mathcal{L}) = 0$. Note that $\mathcal{L}$ has no nonzero global sections since its degree is negative. Applying Riemann-Roch yields that $\deg \mathcal{L} + 1 -g = 0$. This is only possible for $\deg \mathcal{L} = -1$ and $g = 0$. Hence, $Y = \mathbb{P}^1$ and $\mathcal{L} = \mathcal{O}_{\mathbb{P}^1}(-1)$.

The other implication follows from Proposition \ref{paramaz} and from applying Serre duality to $\mathcal{O}_{\mathbb{P}^1}(-1)$ to see that $H^1(\mathbb{P}^1, \mathcal{O}_{\mathbb{P}^1}(-1)) = 0$.
\end{proof}

\begin{Bem}
\label{BemAzLinebundlesPos}
We now point out why Lemma \ref{Lnotamplenaz} is false in positive characteristic. So let $Y$ be a smooth projective curve over an algebraically closed field of characteristic $p > 0$ and let $\mathcal{L}$ be a line bundle on $Y$.

Assume that $\deg \mathcal{L} < 0$ and that $c \in H^1(Y, \mathcal{L})$ is nonzero. Consider the (dual) extension $0 \to \mathcal{O}_Y \to \mathcal{E} \to \mathcal{L}^\vee \to 0$ defined by $c$. As in Lemma \ref{Lnotamplenaz} we have that $\mu_{\min}(\mathcal{E}) > 0$.
If $\mathcal{E}$ is strongly semistable then it is ample as well (\cite[Theorem 2.3]{brennerslope}). If $\mathcal{E}$ is not semistable then the quotients of the Harder-Narasimhan filtration are line bundles (since $\rk \mathcal{E} = 2$). In particular, this is a strong Harder-Narasimhan filtration and it follows that $\mu_{\min}(\mathcal{E}) = \bar{\mu}_{\min}(\mathcal{E}) > 0$ -- and again $\mathcal{E}$ is ample by \cite[Theorem 2.3]{brennerslope}.

So assume now that $\mathcal{E}$ is semistable but not strongly semistable. In this case we have to consider a sufficiently high Frobenius pull back so that the Harder-Narasimhan filtration of ${F^e}^\ast \mathcal{E}$ is strong. The issue is that it may indeed happen that ${F^e}^\ast(c) = 0$ and in this case $\mathcal{E}$  is not ample and $c$ is almost zero. See \cite[Example 3.2]{hartshorneamplecurve} for an explicit case where this happens.
\end{Bem}

\section{Almost zero for vector bundles and slope conditions}

In this section we turn our attention to locally free sheaves of arbitrary rank.

\begin{Le}[Persistence]
\label{Lazkes}
Let $\mathcal{S}, \mathcal{T}$ be locally free sheaves on a smooth projective curve $Y$ over an algebraically closed field with a morphism $\mathcal{S} \to \mathcal{T}$. If $c \in H^1(Y, \mathcal{S})$ is almost zero then its image in $H^1(Y, \mathcal{T})$ is almost zero as well.
\end{Le}
\begin{proof}
Assume that $0 \neq s \in H^0(Y',\mathcal{L})$ annihilates $c$ on $\varphi: Y' \to Y$ such that $\frac{\deg \mathcal{L}}{\deg \varphi} < \eps$. Then we have a commutative square
\[ \begin{xy} \xymatrix{H^1(Y',\varphi^\ast\mathcal{S} \otimes \mathcal{L}) \ar[r] & H^1(Y', \varphi^\ast\mathcal{T} \otimes \mathcal{L}) \\
    H^1(Y',\varphi^\ast\mathcal{S})  \ar[u]_{\cdot s} \ar[r]& H^1(Y',\varphi^\ast\mathcal{T})\ar[u]_{\cdot s}}
   \end{xy}
\] which proves the assertion.
\end{proof}

\begin{Le}
Let $\mathcal{S}$ be a locally free sheaf on a smooth projective curve $Y$ over an algebraically closed field and $c \in H^1(Y, \mathcal{S})$. Then $c$ is almost zero if and only if $\varphi^\ast(c) \in H^1(X, \varphi^\ast \mathcal{S})$ is almost zero for every finite morphism $\varphi: X \to Y$ of smooth projective curves.
\end{Le}
\begin{proof}
The implication from right to left is trivial. Conversely, assume that $c$ is almost zero. That is, for $\eps >0$ there is a finite morphism $\psi:Y' \to Y$ of smooth projective curves and a line bundle $\mathcal{L}$ on $Y'$ with a nonzero global section $s$ such that $s\psi^\ast(c) = 0$ and such that $\deg \mathcal{L}/\deg \psi < \eps$.

Consider the normalisation $Z$ of an irreducible component of the reduced fibre product $X \times_Y Y'$ and note that we have a surjection $\eta: Z \to Y$. In particular, $\eta$ is a finite morphism. And pulling back $\mathcal{L}$ to $Z$ we have that $\varphi^\ast(c)$ is annihilated by the pullback of $s$.

\[
\begin{xy}
 \xymatrix{Z\ar[dr] \\
& X \times_Y Y' \ar[r]^{p_1} \ar[d]_{p_2} & X \ar[d]_{\varphi} \\
& Y' \ar[r]^{\psi} & Y}
\end{xy}
 \]  
\end{proof}

To be able to make finite base changes we need the following stronger notion of almost zero for vector bundles. 

\begin{Def}
Let $Y$ be a smooth projective curve over an algebraically closed field and $\mathcal{S}$ a locally free sheaf on $Y$. We say that $\mathcal{S}$ is \emph{universally almost zero} if for every finite morphism $\varphi: Y' \to Y$ of smooth projective curves $\varphi^\ast \mathcal{S}$ is almost zero.
\end{Def}

\begin{Bem}
Any line bundle of non-negative degree on a smooth projective curve $Y$ over an algebraically closed field is universally almost zero. If $g(Y) \geq 1$ and the characteristic of the base is zero then a line bundle is almost zero if and only if it is universally so. We will see later (Corollary \ref{KoVectorbundleAlmostZero} and Remark \ref{UniversallyAlmostZeroPosChar}) that $\mathcal{O}_{\mathbb{P}^1}(-1)$ on $\mathbb{P}^1$ is essentially the only exception in characteristic zero: A vector bundle on a smooth projective curve of genus $g \geq 1$ over an algebraically closed field of characteristic zero is almost zero if and only if it is universally almost zero.

We do not know if this is true in positive characteristic if $g \geq 2$. It is however true for elliptic curves (see Remark \ref{UniversallyAlmostZeroPosChar}).
\end{Bem}

\begin{Le}
\label{Almostzerotwist}
Let $Y$ be a smooth projective curve over an algebraically closed field. Let $\mathcal{S}$ be a locally free sheaf and $\mathcal{L}$ a line bundle of positive degree. If $\mathcal{S}$ is universally almost zero then so is $\mathcal{S} \otimes \mathcal{L}$.
\end{Le}
\begin{proof}
Making a finite pullback we may assume that $\mathcal{L}$ is generated by global sections by virtue of Corollary \ref{globallygeneratedafterpullback}.
Assume that $\mathcal{S}$ is almost zero. Note that the map $H^1(Y, \mathcal{S}) \to H^1(Y, \mathcal{S} \otimes \mathcal{L})$ induced by $s \in H^0(Y, \mathcal{L})$, $s \neq 0$, is surjective. Indeed, $(\mathcal{S} \otimes \mathcal{L})/\mathcal{S}$ is a torsion sheaf on a curve. Hence, its first cohomology vanishes. Therefore, the claim follows from Lemma \ref{Lazkes}.
\end{proof}

Note that if we assume $\mathcal{L}$ to have a global section then the assertion of the lemma continues to hold if we replace ``universally almost zero'' by ``almost zero''.

\begin{Le}
\label{Lmaeh}
Let $Y$ be a smooth projective curve over an algebraically closed field. Let $0 \to \mathcal{S}' \to \mathcal{S} \to \mathcal{S}'' \to 0$ be a short exact sequence of locally free sheaves. If $\mathcal{S}'$ is universally almost zero and $\mathcal{S}''$ is almost zero then $\mathcal{S}$ is almost zero. Moreover, if $\mathcal{S}'$ and $\mathcal{S}''$ are universally almost zero then $\mathcal{S}$ is universally almost zero.
\end{Le}
\begin{proof}
Fix $c \in H^1(Y, \mathcal{S})$. Consider the image of $c$ in $H^1(Y, \mathcal{S''})$ and annihilate this by a non-trivial global section $s$ of some line bundle $\mathcal{L}$ over $\varphi: Y' \to Y$ such that $0 < \frac{\deg \mathcal{L}}{\deg \varphi} < \frac{\eps}{2}$. It follows that the image of $c$ in $H^1(Y',\varphi^\ast \mathcal{S}'' \otimes \mathcal{L})$ is zero. Hence we find a preimage $c'$ in $H^1(Y', \varphi^\ast\mathcal{S}' \otimes \mathcal{L})$. By Lemma \ref{Almostzerotwist}, $\varphi^\ast\mathcal{S}' \otimes \mathcal{L}$ is almost zero.
Consequently, we find a line bundle $\mathcal{G}$ on some finite curve $\psi: Y'' \to Y'$, where $Y''$ is smooth, with a global section such that $\frac{\deg \mathcal{G}}{\deg \varphi \psi} < \frac{\eps}{2}$ and such that $c'$ is annihilated by a non-trivial global section $t$ of $\mathcal{G}$. It follows that the image of $c$ is zero in $H^1(Y'', \psi^\ast(\varphi^\ast  \mathcal{S} \otimes \mathcal{L}) \otimes \mathcal{G})$. We illustrate the situation with the following commutative diagram where we have omitted the pullbacks.
 
\[ \begin{xy} \xymatrix{H^1(Y,\mathcal{S}') \ar[r] \ar[d]^{\cdot s} & H^1(Y, \mathcal{S}) \ar[d]^{\cdot s} \ar[r] & H^1(Y, \mathcal{S}'') \ar[d]^{\cdot s}\\
    H^1(Y',\mathcal{S}' \otimes \mathcal{L})  \ar[r] \ar[d]^{\cdot t}& H^1(Y',\mathcal{S} \otimes \mathcal{L}) \ar[r] \ar[d]^{\cdot t} & H^1(Y',\mathcal{S}'' \otimes \mathcal{L}) \ar[d]^{\cdot t}\\
	H^1(Y'', \mathcal{S}' \otimes \mathcal{L} \otimes \mathcal{G}) \ar[r]  & H^1(Y'', \mathcal{S} \otimes \mathcal{L} \otimes \mathcal{G}) \ar[r] & H^1(Y'', \mathcal{S}'' \otimes \mathcal{L} \otimes \mathcal{G})}
   \end{xy}
\]
The supplement follows via a similar argument.
\end{proof}


\begin{Prop}
\label{Pmuh}
Let $Y$ be a smooth projective curve over an algebraically closed field $k$. Let $\mathcal{S}$ denote a vector bundle over $Y$ with $\bar{\mu}_{\min}(\mathcal{S}) \geq 0$. Then $\mathcal{S}$ is universally almost zero.
\end{Prop}
\begin{proof}
Note that $\deg \mathcal{S} \geq 0$. We do induction on the rank $n$ of $\mathcal{S}$. For $n = 1$ the result follows by Proposition \ref{paramaz}.
So assume that $n > 1$. By \cite[Theorem 2.3]{brennerslope}, $\mathcal{S}^\vee$ is not ample. Let $\eps > 0$, also by \cite[Theorem 2.3]{brennerslope} we find a finite morphism of smooth curves $\varphi: Y' \to Y$ together with a line bundle $\mathcal{L}$ on $Y'$ such that $\frac{\deg \mathcal{L}}{\deg \varphi} < \frac{\eps}{3}$ and an exact sequence $\varphi^\ast \mathcal{S}^\vee \to \mathcal{L} \to 0$. Dually this yields an exact sequence $0 \to \mathcal{L}^\vee \to \varphi^\ast \mathcal{S} \to \mathcal{G} \to 0$ for some locally free sheaf $\mathcal{G}$ on $Y'$. Note that $\bar{\mu}_{\min}(\mathcal{G}) \geq 0$. If $\deg \mathcal{L}^\vee \geq 0$ then $\mathcal{L}^\vee$ is universally almost zero and applying the induction hypothesis to $\mathcal{G}$ and then Lemma \ref{Lmaeh} to the exact sequence we are done. 

So assume that $\deg \mathcal{L}^\vee < 0$. Choose $\mathcal{M} \in \Pic Y'$ such that $0 < \frac{\deg \mathcal{L}^\vee \otimes \mathcal{M}}{\deg \varphi} < \frac{\eps}{3}$ (e.\,g.\ $\mathcal{M} = \mathcal{L}^2$ will do). In particular, $\mathcal{M}$ is ample since $\deg \mathcal{M} > 0$. By Corollary \ref{globallygeneratedafterpullback} we may therefore assume that $\mathcal{M}$ is generated by global sections. Tensoring with $\mathcal{M}$ we obtain an exact sequence \[\begin{xy}\xymatrix{0 \ar[r]& \mathcal{L}^\vee \otimes \mathcal{M} \ar[r]& \varphi^\ast \mathcal{S} \otimes \mathcal{M} \ar[r]& \mathcal{G} \otimes \mathcal{M} \ar[r]& 0.}\end{xy}\]Therefore as in the first case $\varphi^\ast \mathcal{S} \otimes \mathcal{M}$ is universally almost zero. Since $\mathcal{M}$ has a nonzero global section $s$ we have an induced morphism $\varphi^\ast \mathcal{S} \to \varphi^\ast \mathcal{S} \otimes \mathcal{M}$. Annihilating $H^1(Y',\varphi^\ast \mathcal{S} \otimes \mathcal{M})$ by a non-trivial global section $t$ of some line bundle $\mathcal{N}$ on a smooth curve $\psi: Y'' \to Y'$ finite over $Y'$ such that $\deg \mathcal{N}/\deg \psi < \frac{\eps}{3}$ yields that $\varphi^\ast \mathcal{S}$ is almost zero. Indeed, $\psi^\ast(s) \otimes t \in H^0(Y'', \psi^\ast \mathcal{M} \otimes \mathcal{N})$ annihilates every $c \in H^1(Y', \varphi^\ast \mathcal{S})$ and $\deg \psi^\ast \mathcal{M} \otimes \mathcal{N}/\deg \psi < \eps$.

By the same token, $\mathcal{S}$ is universally almost zero.
\end{proof}

\begin{Bem}
\label{BemLmaehUmkehrung}
The converse of Lemma \ref{Lmaeh} is false. To see this consider the sheaf of differentials on $Y =\mathbb{P}^1_k$ which is isomorphic to $\mathcal{O}_Y(-2)$, where $k$ is an algebraically closed field. Twisting its presenting sequence by $\mathcal{O}_Y(1)$ we have an exact sequence \[0 \to \mathcal{O}_Y(-1) \to \mathcal{O}_Y \oplus \mathcal{O}_Y \to \mathcal{O}_Y(1) \to 0  \quad \text{(cf.\ \cite[Theorem II.8.13]{hartshornealgebraic})}.\]
Consider the homomorphism $k[x,y] \to k[s,t]$ defined by $x \mapsto s^d, y \mapsto t^d$. This is a finite injective homomorphism and if we attach to both rings the ordinary grading we obtain a finite dominant morphism $\varphi_d: Y' \to Y$ such that $\varphi^\ast\mathcal{O}_Y(1) = \mathcal{O}_{Y'}(d)$ by virtue of Lemma \ref{Pullback}. Note that both $Y'$ and $Y$ are isomorphic to $\mathbb{P}^1_k$.
Pulling back along $\varphi_d$ ($d >1$) one obtains that both $\mathcal{O}_{Y'} \oplus \mathcal{O}_{Y'}$ and $\mathcal{O}_{Y'}(d)$ are strongly semistable of degree $\geq 0$, hence universally almost zero. But $\mathcal{O}_{Y'}(-d)$ is strongly semistable of negative degree. Hence, it is not almost zero if $\chara k = 0$ since $d > 1$ (cf.\ Theorem \ref{Plbaz}). If the characteristic of the base is positive this counterexample continues to hold. Indeed, denote by $\mathcal{E}$ the extension defined by the class $c$. Then $0 \to \mathcal{O}_{Y'}(-d) \to \mathcal{E} \to \mathcal{O}_{Y'} \to 0$ is a strong Harder-Narasimhan filtration of $\mathcal{E}$ and $\mathcal{O}_{Y'}(-d)$ is not almost zero by Remark \ref{BemAzLinebundlesPos}.
\end{Bem}

\begin{Ko}
\label{almostzerostronglysemistablecase}
Let $Y$ be a smooth projective curve over an algebraically closed field $k$ and $\mathcal{S}$ a strongly semistable vector bundle. Then $c \in H^1(Y, \mathcal{S})$ is not almost zero if and only if $\deg \mathcal{S} < 0$ and $c \neq 0$ (in positive characteristic ${F^e}^\ast(c) \neq 0$ for all Frobenius powers $F^e$).
\end{Ko}
\begin{proof}
Assume that $c$ is not almost zero. Thus Proposition \ref{Pmuh} implies that $\deg \mathcal{S} < 0$ and we must clearly have that ${F^e}^\ast(c) \neq 0$ for all Frobenius powers $F^e$.

For the other direction assume that $\deg \mathcal{S} < 0$ and that ${F^e}^\ast(c) \neq 0$ for all $e$.  Since $\bar{\mu}_{\min}(\mathcal{S}^\vee) > 0$, we have that $\mathcal{S}^\vee$ is ample by \cite[Theorem 2.3]{brennerslope}. Denote by $\mathcal{S}'$ the extension of $\mathcal{O}_Y$ by $\mathcal{S}$ defined by $c$. Since $c$ defines a non-trivial extension, it follows from \cite[Proposition 2.2]{giesekerample} that every quotient of $\mathcal{S}'^\vee$ has positive degree. As the extension does stay non-trivial for all Frobenius pullbacks we have $\bar{\mu}_{\min}(\mathcal{S}'^\vee) > 0$. Hence, again by \cite[Theorem 2.3]{brennerslope} we have that $\mathcal{S}'^\vee$ is ample and therefore $c$ is not almost zero by Theorem \ref{Lazas}.
\end{proof}

Now we can finally prove
\begin{Theo}
\label{Daggersolidsst}
Let $R$ be a normal standard graded two-dimensional integral $k$-algebra, where $k$ is an algebraically closed field of arbitrary characteristic. Moreover, let $I$ be an $R_+$-primary homogeneous ideal with homogeneous generators $f_1, \ldots, f_n$ and assume that $\Syz(f_1, \ldots, f_n)$ on $\Proj R$ is strongly semistable. Then we have 
\[(f_1, \ldots, f_n)^{\dagger \text{GR}} = (f_1, \ldots, f_n)^\star. \]
\end{Theo}
\begin{proof}
Combine Theorem \ref{almostzerodagger}, Corollary \ref{almostzerostronglysemistablecase} and \cite[Proposition 2.1]{brennertightplus}.
\end{proof}

\section{Geometric reductions}
\label{Reductions}
Assume Situation \ref{Sit1}. We now want to reduce the issue whether $(f_1, \ldots, f_n)^\star = (f_1, \ldots, f_n)^{\dagger \text{GR}}$ to the strongly semistable case using a strong Harder-Narasimhan filtration. We will need to look at the cohomology class $c$ defined by the image of $f_0$ via the connecting homomorphism $H^0(Y,\mathcal{O}_Y(d_0)) \to H^1(Y,\mathcal{S})$.

In \cite{brennertightplus} this reduction along a strong Harder-Narasimhan filtration is carried out for solid closure and we will follow the arguments there suitably adopted to our situation.
Let $\mathcal{S}$ be a vector bundle on a smooth projective curve over an algebraically closed field $k$ and $\mathcal{S}_1 \subset \ldots \subset \mathcal{S}_t = {F^e}^\ast \mathcal{S}$ a strong Harder-Narasimhan filtration.
We will need to look at the maximal $i$ such that $\mu(\mathcal{S}_i/\mathcal{S}_{i-1}) \geq 0$. If $\mu(\mathcal{S}_j/\mathcal{S}_{j-1}) < 0$ for all $j=1, \ldots, t$ then set $i = 0$ and $\mathcal{S}_0 = \mathcal{S}_{-1} = 0$ and if $\mu(\mathcal{S}_j/\mathcal{S}_{j-1}) \geq 0$ for all $j=1, \ldots, t$ then set $i = t+1$ and $\mathcal{S}_{t+1} = \mathcal{S}_t = {F^e}^\ast\mathcal{S}$. 
We recall that if the characteristic of the field is zero then the Frobenius is replaced by the identity.


\begin{Theo}
\label{HNFaztight}
Let $Y$ be a smooth projective curve over an algebraically closed field $k$. Let $\mathcal{S}$ be a locally free sheaf on $Y$ and $c \in H^1(Y, \mathcal{S})$. Let $\mathcal{S}_1 \subset \ldots \subset \mathcal{S}_t = {F^e}^\ast\mathcal{S}$ be a strong Harder-Narasimhan filtration on $Y$. Choose $i$ such that $\mathcal{S}_i/\mathcal{S}_{i-1}$ has degree $\geq 0$ and such that $\mathcal{S}_{i+1}/\mathcal{S}_i$ has degree $<0$. Let $0 \to \mathcal{S}_i \to {F^e}^\ast \mathcal{S} \to {F^e}^\ast \mathcal{S}/\mathcal{S}_i = \mathcal{Q} \to 0$. Then the following are equivalent:
\begin{enumerate}[(i)]
\item{The class $c$ is almost zero.}
\item{Some Frobenius power of the image of ${F^e}^\ast(c)$ in $H^1(Y, \mathcal{Q})$ is zero.}
\end{enumerate}
\end{Theo}
\begin{proof}
Assume that the image of ${F^e}^\ast(c)$ in $H^1(Y, \mathcal{Q})$ is nonzero for all $e$.
Fix an $e$ and let $c'$ be the image of ${F^e}^\ast(c)$ in  $H^1(Y, \mathcal{Q})$. Then it follows as in the proof of Corollary \ref{almostzerostronglysemistablecase} that $\mathcal{Q}'^\vee$ is an ample vector bundle. Hence, Theorem \ref{Lazas} yields that $c'$ is not almost zero. Consequently, by Lemma \ref{Lazkes}, $c$ is not almost zero.

Suppose now that (ii) holds. We may assume that ${F^e}^\ast(c)$ is $0$ in $H^1(Y, \mathcal{Q})$. Thus ${F^e}^\ast(c)$ stems from a cohomology class $c_i$ in $H^1(Y, \mathcal{S}_i)$. Now the result follows from Propositions \ref{Pmuh} and \ref{Lazkes}. Indeed, $\mathcal{S}_1 \subset \ldots \subset \mathcal{S}_i$ is a strong Harder-Narasimhan filtration of $\mathcal{S}_i$. Therefore, $\bar{\mu}_{\min}(\mathcal{S}_i) \geq 0$. 
\end{proof}

\begin{Ko}
\label{KoVectorbundleAlmostZero}
Let $Y$ be a smooth projective curve of genus $g \geq 1$ over an algebraically closed field of characteristic zero and $\mathcal{S}$ a locally free sheaf on $Y$. Let $\mathcal{S}_1 \subset \ldots \subset \mathcal{S}_t$ be the Harder-Narasimhan filtration of $\mathcal{S}$ on $Y$. Choose $i$ such that $\mathcal{S}_i/\mathcal{S}_{i-1}$ has degree $\geq 0$ and that $\mathcal{S}_{i+1}/\mathcal{S}_i$ has degree $<0$. Let $0 \to \mathcal{S}_i \to \mathcal{S} \to \mathcal{S}/\mathcal{S}_i = \mathcal{Q} \to 0$.
Then the following are equivalent:
\begin{enumerate}[(i)]
\item{$\mu_{\min}(\mathcal{S}) \geq 0$.}
\item{$\mathcal{S}$ is almost zero.}
\item{$\mathcal{S}$ is universally almost zero.}
\item{$\mathcal{Q} = 0$.}
 \item{$H^1(Y, \mathcal{Q}) = 0$.} 
\end{enumerate}
\end{Ko}
\begin{proof}
The equivalence of (v) and (ii) is immediate from Theorem \ref{HNFaztight} since the map $H^1(Y, \mathcal{S}) \to H^1(Y, \mathcal{Q})$ is surjective. Likewise, the equivalence of (i) and (iv) is immediate from the definition of $\mathcal{Q}$. Assume (v) and assume that $\mathcal{Q} \neq 0$. Note that we have $H^0(Y, \mathcal{Q}) = 0$ since $\mu_{\max}(\mathcal{Q}) < 0$. We must have that $\deg \mathcal{Q} = \chi(\mathcal{Q}) - \rk \mathcal{Q} \, \chi(\mathcal{O}_Y) = \rk \mathcal{Q} (g -1) \geq 0$ -- a contradiction. Thus $\mathcal{Q} = 0$.
The implication from (i) to (iii) follows from Proposition \ref{Pmuh}. Finally, assume (iii). Then taking for $\varphi$ the identity shows that (ii) holds.
\end{proof}

\begin{Bem}
\label{UniversallyAlmostZeroPosChar}
\begin{enumerate}[(a)]
\item{We need to exclude $\mathbb{P}^1$ in Corollary \ref{KoVectorbundleAlmostZero} since $\mathcal{O}_{\mathbb{P}^1}(-1)$ is almost zero of negative degree. The equivalence of (i) and (ii) still holds. But (iii), (iv) and (v) are no longer satisfied. Indeed, this is clear for (iii), (iv) and pulling back along a morphism $\varphi_d$ ($d> 1$) as in Remark \ref{BemLmaehUmkehrung} contradicts (v).}
\item{The statement of the Corollary remains valid in arbitrary characteristic if $Y$ is elliptic. This follows from a theorem of Oda (see e.\,g.\ \cite[Theorem 2.2]{brennertightelliptic} for a more general version) which asserts that the map $F^\ast: H^1(Y, \mathcal{S}) \to H^1(Y', F^\ast \mathcal{S})$ is injective, where $F$ is the (relative) Frobenius and $\mathcal{S}$ a locally free sheaf whose indecomposable components are of negative degree. Note that this last condition on $\mathcal{S}$ is equivalent to $\bar{\mu}_{\max}(\mathcal{S}) = \mu_{\max}(\mathcal{S}) < 0$ by \cite[Theorem 1.3]{hartshorneamplecurve}. Applying this to $\mathcal{Q}$ in Corollary \ref{KoVectorbundleAlmostZero} yields the equivalence of (ii) and (v).

We do not know whether the corollary holds true in general in positive characteristic. There are of course cases, where the Frobenius is not injective in cohomology for vector bundles of degree $< 0$ (see e.\,g.\ \cite[Example 3.2]{hartshorneamplecurve}). But for the corollary to be false one would need a map $H^1(Y, \mathcal{Q}) \to H^1(Y', {F^e}^\ast \mathcal{Q})$ such that the Frobenius is identically zero and $\bar{\mu}_{\max}(\mathcal{Q}) < 0$.}
\end{enumerate}
\end{Bem}

\begin{Theo}
\label{mainprimary} 
Let $k$ be an algebraically closed field. Let $R$ be a standard graded two-dimensional normal domain of finite type over $k$. Then for every homogeneous $R_+$-primary ideal $I$ we have
\[I^\star = I^{\dagger \text{GR}}. \]
\end{Theo}
\begin{proof}
In light of Lemma \ref{daggerhomogeneous0} and Corollary \ref{solidhomogeneous} we may restrict our attention to homogeneous elements. So let $I = (f_1, \ldots, f_n)$ and let $f_0$ be homogeneous with corresponding cohomology class $\delta(f_0) = c$ and torsor $T$.
Combining \cite[Theorem 2.3]{brennertightplus} and Theorem \ref{HNFaztight} we have that the torsor $T$ is not affine if and only if the corresponding cohomology class $c$ is almost zero. The non-affineness of $T$ is equivalent to containment in solid closure by \cite[Proposition 3.9]{brennertightproj}. And $c$ is almost zero if and only if the element $f_0$ is contained in graded dagger closure by Theorem \ref{almostzerodagger}.
\end{proof}

\section{Algebraic reductions}
\label{AlgebraicReductions}

In this section we prove Theorem \ref{mainprimary} without the conditions standard graded, normal, $R_+$-primary and $k$ algebraically closed. In doing this, we will frequently need to pass from $R$ to a finite graded extension domain $S$. It is then clear from the definition of graded dagger closure that if an element $f$ of $R$ is in $(IS)^{\dagger \text{GR}}$ it is also contained in $I^{\dagger \text{GR}}$. For solid closure we recall the following

\begin{Prop}
\label{hochster2}
Assume that $R \to S$ is a finite extension of noetherian domains, $I \subseteq R$ is an ideal and let $f \in R$. Then $f \in (IS)^\star$ implies $f \in I^\star$.
\end{Prop}
\begin{proof}
These conditions imply that \cite[Theorem 5.9 (c)]{hochstersolid} is satisfied -- see \cite[Remarks 1.7.6]{hochsterhuneketightzero} for the argument.
\end{proof}

Before presenting the first reduction result we need two somewhat technical lemmata.
We shall also need the notion of a \emph{paraclass} in the next lemma. Let $R$ be a $d$-dimensional $\mathbb{N}$-graded domain finitely generated over a field $R_0$. Let $(x_1, \ldots, x_d)$ be homogeneous parameters for $R$. This yields an element $1/(x_1 \cdots x_d) \in H^d_{R_+}(R)$. Any such element is called a \emph{canonical element} or a \emph{paraclass}. Since $R$ contains a field such a class is nonzero (see \cite[Theorem 9.2.1 and Remark 9.2.4 (b)]{brunsherzog}). Moreover, if $A$ is a forcing algebra and $d =2$ then $H^2_{R_+}(A) = 0$ if and only if some (equivalently every) paraclass coming from $R$ vanishes (see \cite[Proposition 1.9]{brennerparasolid}). This is \emph{not} true if $d \geq 3$ in equal characteristic zero, and indeed, this is the issue which parasolid closure addresses. We refer to \cite[Section 1]{brennerparasolid} for an elaborate discussion of paraclasses and the connection to (para)solid closure and also to \cite[Sections 9.2 and 9.3]{brunsherzog} for further discussion of paraclasses.

Finally, we recall that the vanishing of a paraclass $c = 1/(x_1 \cdots x_d)$ in $H^d_{R_+}(A)$ is equivalent to $(x_1 \cdots x_d)^t \in (x_1^{t+1}, \ldots, x_d^{t+1})$ in $A$ for some $t \in \mathbb{N}$ (see \cite[Remark 9.2.4 (b) and the discussion at the beginning of Section 9.3]{brunsherzog}).

\begin{Le}
\label{paraclassvansihing}
Let $R$ be a two-dimensional domain of finite type over a field $R_0 = k$, $I \subseteq R$ a homogeneous ideal and $f$ a homogeneous element of $R$. Let $A$ be the forcing algebra for $(f, I)$ and assume that $H^2_{R_+}(A) = 0$. Then there exists a homogeneous $R_+$-primary ideal $J$ containing $I$ with forcing algebra $A'$ for $(f, J)$ such that $H^2_{R_+}(A') = 0$.
\end{Le}
\begin{proof}
Let $I= (f_1, \ldots, f_n)$ and $f \in R$. Write \[R[T_1, \ldots, T_n]/(\sum_{i=1}^n f_i T_i - f)\] for the forcing algebra $A$ and assume that $H^2_{R_+}(A) = 0$. In particular, paraclasses vanish, that is, we have a relation \[(xy)^t = a_1 x^{t+1} + a_2 y^{t+1} + P (\sum_{i=1}^n f_i T_i -f) \text{ in } R[T_1, \ldots, T_n],\]where $x,y$ are homogeneous parameters for $R_+$ and $a_1, a_2, P \in R[T_1, \ldots, T_n]$.

Consider $J = I + (x^{t+1}, y^{t+1})$ -- this is obviously $R_+$-primary and contains $I$. A forcing algebra for $(f, J)$ is given by \[A' = R[T_1, \ldots, T_n, U_1, U_2]/(\sum_{i=1}^n f_i T_i + x^{t+1} U_1 + y^{t+1} U_2 -f).\]In $R[T_1, \ldots, T_n, U_1, U_2]$ we obtain the equation \[(xy)^t = (a_1 - P U_1) x^{t+1} + (a_2 - P U_1) y^{t+1} + P (\sum_{i=1}^n f_i T_i + x^{t+1} U_1 + y^{t+1} U_2  -f).\]This means that the paraclass $1/(xy)$ vanishes in $H^2_{R_+}(A')$ and since $R$ has dimension two this implies that $H^2_{R_+}(A') = 0$.
\end{proof}

\begin{Le} 
\label{irrelevantideal}
Let $R$ be a $\mathbb{N}$-graded domain of dimension two that is finitely generated over a field $R_0$. Let $I \subseteq R$ be a homogeneous ideal and $f \in R$. Then $f \in I^\star$ if and only if $f \in (IR_{R_+})^\star$.
\end{Le}
\begin{proof}
The only if part is clear by the persistence of solid closure (cf.\ \cite[Theorem 5.6]{hochstersolid}). If $\height I$ is $0$ or $2$ then the assertion is also clear. For $\height I = 0$ implies $I = 0$ and $0^\star = \Rad 0 = 0$. If $\height I = 2$ then $R_+$ is the only maximal ideal containing $I$ (this follows as the minimal primes over $I$ are homogeneous -- see \cite[Lemma 1.5.6 (a)]{brunsherzog}). Since we only have to consider the completions at maximal ideals containing $I$ the claim follows.
 
So we may assume that $\height I = 1$.
By Proposition \ref{hochster2} we may pass to finite graded ring extensions. Adjoining roots of generators of $R$ we may assume that $R$ is standard graded (hence $\mathcal{O}_{\Proj R}(1)$ is locally free and generated by global sections) so that we may work with extensions that are section rings. Furthermore, by passing to the section ring corresponding to $\varphi^\ast \mathcal{O}_{\Proj R}(1)$, where $\varphi$ denotes the normalisation morphism, we may assume by \cite[Proposition 2.1 (9)]{hyrysmith} that $R$ has an isolated normal singularity in $R_+$. Since $\varphi^\ast \mathcal{O}_{\Proj R}(1)$ is still generated by global sections we have by Proposition \ref{PCover} that $\Proj R$ is covered by standard open sets coming from elements of degree $1$.

Applying Corollary \ref{solidhomogeneous} we may assume $f$ to be homogeneous.
We have that $f \in (IR_{P})^\star = I R_P$ for every minimal prime $P$ over $I$ by persistence and since ideals in regular rings of dimension $\leq 2$ are solidly closed. Fix a minimal prime $P$ over $I$. We then have $uf \in I$ for some $u \in R$ that is not contained in $P$. And we may assume $u$ to be homogeneous since $I, P$ and $f$ are homogeneous.

Let $P_1, \ldots, P_n$ be the minimal primes over $I$. We then have homogeneous elements $u_i \in R \setminus P_i$ such that $u_i f \in I$. Moreover, we have elements $s_1, \ldots, s_m$ of degree $1$ that cover $\Proj R$. This implies ${s_j}^{- \deg f} f \in (IR_{P_i})_0$ for all $i$ and suitable $j$. 
Furthermore, the $D_+(u_i s_j) \cap V_+(I)$ cover $V_+(I) = \{P_1, \ldots, P_n\}$. Looking at the cone mapping we see that the $D(u_i s_j) \cap V(I)$ cover $V(I) \setminus R_+$. Hence, $f \in IR_M = (IR_M)^\star$ for any maximal ideal $M \neq R_+$. Since, by assumption, $f \in (IR_{R_+})^\star$ it follows that $f \in (I R_M)^\star$ for every maximal ideal $M$ of $R$. Hence, $f \in I^\star$.
\end{proof}

\begin{Theo}
\label{TRed}
Let $k$ be an algebraically closed field. Let $R$ be an $\mathbb{N}$-graded two-dimensional domain of finite type over $R_0 = k$. Then for every homogeneous ideal $I$ we have
\[I^\star = I^{\dagger \text{GR}}. \]
\end{Theo}
\begin{proof}
We first reduce to the primary case using Theorem \ref{mainprimary}. So suppose in addition that $R$ is normal and standard graded. Let $I = (f_1, \ldots, f_n)$ be a homogeneous ideal. Suppose $f \in I^\star$ for some $f \in R$. For every $l \in \mathbb{N}$ we have $f \in (I + R_{\geq l})^\star$ and we may assume $f$ to be homogeneous of degree $m$ due to Corollary \ref{solidhomogeneous}. Since these ideals are $R_+$-primary we have $f \in (I + R_{\geq l})^{\dagger \text{GR}}$ by Theorem \ref{mainprimary}. Therefore, by Lemma \ref{Ldeg} we have  for $r \in \mathbb{N}$ a nonzero element $a_r$ of degree $\frac{1}{r}$ (which may depend on $l$ but the degree does not) in some finite $\mathbb{Q}$-graded extension domain $S$ of $R$ such that $a_r f = \sum_i s_i f_i + \sum t_j g_j$ with $s_i, t_j \in S$ and $g_j \in R_{\geq l}$. We may assume that everything is homogeneous, hence for $l > m+1$ we get $t_j = 0$ and therefore $a_rf \in IS$.

Suppose now that $f \in I^{\dagger \text{GR}}$ and assume that $f \notin I^\star$. By Lemma \ref{irrelevantideal} this happens if and only if $H^2_{R_+}(A) = 0$, where $A$ is the forcing algebra for $(f, I)$ (note that since $R$ is normal and excellent $R' = \widehat{R_{R_+}}$ is integral, and $H^d_{\widehat{R_+}}(R' \otimes_R A) = H^d_{R_+}(A)$ by flat base change \cite[Theorem 4.3.2]{brodmannsharp}, since $H^d_m(A) = H^d_m(R) \otimes A$ and because $H^d_{\widehat{m}}(\widehat{R}) = H^d_m(R)$). 
By Lemma \ref{paraclassvansihing} it follows that there exists an $R_+$-primary homogeneous ideal $I \subseteq J$ such that $f \notin J^\star$. But this is a contradiction since we must have $f \in J^{\dagger \text{GR}} = J^\star$ by Theorem \ref{mainprimary}.

Assume now that $R$ is a two-dimensional $\mathbb{N}$-graded domain of finite type over an algebraically closed field $k = R_0$. Write $R = k[x_1, \ldots, x_r]/P$ with $\deg x_i = e_i$ and adjoin $e_i$th roots of the $x_i$ (cf.\ Lemma \ref{stdgraded}). Call the normalisation of this ring $S$. We therefore have a finite injective mapping $R \to S$ such that the $D_+(s), s \in S_1$, cover $\Proj S$. Note that Theorem \ref{mainprimary} still holds under this weaker hypothesis.

By \cite[Theorem 5.6]{hochstersolid} we have that $f \in I^\star$ implies $f \in (IS)^\star$. But by Theorem \ref{mainprimary} the containment $f \in (IS)^\star$ yields that $f \in (IS)^{\dagger \text{GR}}$ and then $f \in  I^{\dagger \text{GR}}$.
For the converse suppose that $f \in I^{\dagger \text{GR}} \subseteq (IS)^{\dagger \text{GR}}$. Hence, we have $f \in (IS)^\star$ by Theorem \ref{mainprimary}. Since $R \subseteq S$ is finite Proposition \ref{hochster2} implies $f \in I^\star$.
\end{proof}

\begin{Theo}
\label{Mainresult}
Let $R$ denote an $\mathbb{N}$-graded two-dimensional domain of finite type over a field $R_0$ and $I$ a homogeneous ideal of $R$. Then $I^{\dagger \text{GR}} = I^\star$. 
\end{Theo}
\begin{proof}
We may assume that $R$ is normal. Furthermore, $R$ is geometrically integral. Indeed, this is the case if and only if $Q(R) \cap k = R_0$, where $k$ denotes an algebraic closure of $R_0$ by virtue of \cite[Corollary 3.2.14 (c)]{liualgebraicgeometry}. And elements of $Q(R) \cap k$ are integral over $R_0$ and hence contained in $R$ since $R$ is normal. Moreover, as any such nonzero element is a unit it is necessarily contained in $R_0$.

Thus, we may identify $R_k = R \otimes_{R_0} k$ with $R[\alpha \, \vert \, \alpha \in \overline{Q(R)} \text{ is algebraic over } R_0]$, where $\overline{Q(R)}$ is an algebraic closure of the field of fractions of $R$.
Let $f \in I^{\dagger \text{GR}}$. It follows that $f \in (IR_k)^{\dagger \text{GR}} = (IR_k)^\star$ by Theorem \ref{TRed}. Since $R_k$ is faithfully flat over $R$ we have by \cite[Theorem 5.9 (a)]{hochstersolid} that $f \in I^\star$.

For the converse assume that $f \in I^\star$. The persistence of solid closure \cite[Theorem 5.6]{hochstersolid} implies that $f \in (IR_k)^\star$. And the latter is equal to $(IR_k)^{\dagger \text{GR}}$ again by Theorem \ref{TRed}. Therefore, we immediately have $f \in I^{\dagger \text{GR}}$.
\end{proof}

\bibliography{bibliothek.bib}
\bibliographystyle{amsplain}

\end{document}